    \documentclass[11pt]{article}
    \usepackage{amsmath,amsthm,amscd,amssymb,latexsym,verbatim, xfrac}
    \usepackage{amssymb}

\usepackage[normalem]{ulem}

\usepackage{graphicx,epsf,cite}

\usepackage[shortlabels]{enumitem}
    
\usepackage{xcolor}
\usepackage{graphicx}

    \usepackage[margin=1.2in]{geometry}

    \newtheorem{theorem}{Theorem}[section]

    \newtheorem{lemma}[theorem]{Lemma}

    \theoremstyle{definition}

    \newtheorem{remark}[theorem]{Remark}
    \newtheorem{definition}[theorem]{Definition}
    \newtheorem{assumption}[theorem]{Assumption}

    \newcounter{smalllist}

    \DeclareMathOperator*{\Lip}{Lip}

    \allowdisplaybreaks
    \numberwithin{equation}{section}
    
    \newcommand{\lb}{\label}

    \newcommand{\wti}{\widetilde  }

    \newcommand{\tr}{\text{\rm{Tr}}}

    \newcommand{\beq}{\begin{equation}}
    \newcommand{\eeq}{\end{equation}}
    
    \newcommand{\bal}{\begin{align}}
    \newcommand{\eal}{\end{align}}
    \newcommand{\bals}{\begin{align*}}
    \newcommand{\eals}{\end{align*}}
    
    \newcommand{\R}{\ensuremath{\mathbb{R}}}

    \newcommand{\bbR}{{\mathbb{R}}}

    \newcommand{\bbS}{{\mathbb{S}}}

    \newcommand{\calS}{{\mathcal S}}

    \newcommand{\calB}{{\mathcal B}}

    \newcommand{\calC}{{\mathcal C}}

    \newcommand{\calK}{{\mathcal K}}

    \newcommand{\ep}{\varepsilon}
    \newcommand{\eps}{\varepsilon}

    \newcommand{\tp}{\tilde{p}}

\date{\today}
\title{Well-posedness of a Hele-Shaw problem in general dimensions}

\author{Olga Turanova\thanks{Department of Mathematics, Michigan State University, Wells Hall, 619 Red Cedar Road,
East Lansing, MI 48824, \texttt{turanova@msu.edu}.}, Yuming Paul Zhang\thanks{Department of Mathematics and Statistics, Auburn University, Parker Hall, 221 Roosevelt Concourse, Auburn, AL 36849, \texttt{yzhangpaul@auburn.edu}.}
}

\begin{document}
     \maketitle
\begin{abstract} 
We investigate a Hele-Shaw type free boundary problem in general spatial dimension. We establish a comparison principle and well-posedness for the problem using a notion of viscosity flows. An important feature of our work is that our hypotheses allow for the free boundary velocity to take negative values.

\end{abstract}
    
\medskip

\noindent {\bf Keywords:} Hele-Shaw type flow, well-posedness, comparison principle, viscosity solutions 

\medskip

\noindent{\bf AMS Subject Classification (2020):}  	 35B51, 35R35, 76D27

\section{Introduction} \lb{S1}

In 1898, H. S. Hele-Shaw introduced what is now known as the classical Hele-Shaw problem to describe the dynamics of an incompressible viscous Newtonian fluid between two narrowly separated parallel surfaces \cite{hele1898experiments}. In this original setting, the fluid pressure $p$ is harmonic within its positive set, and the free boundary expands with a velocity proportional to the pressure gradient. In recent decades, problems of Hele-Shaw type have been used in biological and social modeling, including to describe the mechanical evolution of tumor cells \cite{PQV, David_S, jacobs2022tumor, SulakTuranova} and the dynamics of congested populations \cite{maury2010, CKY}. Furthermore, Hele-Shaw type problems are mathematically related to the one-phase Muskat problem, particularly in the presence of constant gravity \cite{Hongjie, schwab2024well}.

In this paper, we consider the following free boundary problem:
\begin{equation}\lb{main}
\left\{
\begin{aligned}
    &F(x,p,\nabla p,D^2 p)=0 &&\quad\text{ in }\{p>0\},\\
  & \partial_t p=V(x,  \nabla_x p)| \nabla_x p| &&\quad\text{ on }\partial\{p>0\}.
\end{aligned}   
\right.
\end{equation}
Here,  $F$ is an elliptic operator and $V$ is a given velocity function; both  are assumed to satisfy certain regularity conditions (see Assumptions \ref{assump2} and \ref{assump3}). The boundary condition in \eqref{main} dictates that the free boundary $\partial \{p>0\}$ of the unknown function $p(x,t)$ evolves according to the normal velocity $V(x,\nabla_x p)$. 
Our assumptions allow for $V$ to take negative values, which corresponds to contraction of the interface. This is in contrast to the classical case, in which the free boundary can only expand.


It is known that even the simplest constant-coefficient versions of \eqref{main} may develop singularities in finite time \cite{kim2003}. Consequently, the search for appropriate weak solutions has led to several frameworks, including the viscosity solution approach introduced by Kim in \cite{kim2003} and the generalized set solutions developed by Cardaliaguet and co-authors in \cite{card2004, card2007}. The authors' previous work  \cite{turanova2025hele} builds on both of these approaches to introduce a notion of viscosity flows to inhomogeneous problems of Hele-Shaw type in spatial dimension one. 

In the present paper, we define viscosity flows and viscosity solutions for \eqref{main} and establish how these two notions are related (see Definition \ref{D.1.2}, Definition \ref{def23}, and Lemma \ref{L.def}). 
Our first main result is a comparison principle for viscosity flows. 
\begin{theorem}[Comparison principle]\lb{L.cp}
Under Assumptions \ref{assump2} and \ref{assump3}, let two bounded sets $\Omega_1$ and $\Omega_2$  be, respectively,  viscosity sub- and superflows to \eqref{main} 
in $\R^d \times (0,T)$ such that $\overline{\Omega_1(0)}\subset\Omega_2(0)$. Then $\Omega_1\subseteq\Omega_2$. 
\end{theorem}

We use the comparison principle, along with a Perron's method argument, to establish the existence of maximal viscosity flows. Namely, we prove that for any   initial set that is regular enough, there exists a  flow that contains any subflow that starts from the same initial set.
\begin{theorem}[Existence of maximal flows]\lb{T.3.4}
Under Assumptions \ref{assump2} and \ref{assump3}, let $\Omega_0\subseteq\R^d $ be open and bounded and satisfy the interior ball condition. Then there exists an upper semicontinuous set $\Omega\subseteq\R^d\times [0,\infty)$ that is a  viscosity flow of the equation \eqref{main} such that $\Omega(0)=\Omega_0$. Moreover, it contains any other viscosity subflows of \eqref{main} with initial set $\subseteq \Omega_0$.
\end{theorem}
 Finally, we prove two well-posedness results, which, for the purposes of introducing our work and discussing its significance, we describe  informally here:
\begin{itemize}
    \item Theorem \ref{thm:shrink/expand} states that if the free boundary velocity at the initial time is strictly positive or strictly negative, then the maximal flow given by Theorem \ref{T.3.4} is unique.
    \item Theorem \ref{thm:generic} is a generic uniqueness result, which states, roughly, that for ``most" initial sets, there exists a unique viscosity flow starting from that set.
\end{itemize}

The first key feature of our work is that we are able to establish well-posedness of \eqref{main} without assuming that the velocity $V$ is positive. To our knowledge, the only other result on Hele-Shaw type free boundary problems with possibly negative velocity is Kim's \cite{contact}. The second is that we rigorously  connect viscosity solutions defined in terms of ``testing by sets" and ``testing by functions." Such a connection is not available between  viscosity solutions to Hele-Shaw type problems introduced by Kim  and the generalized set solutions of Cardaliaguet and co-authors. Third, by employing the ``generic uniqueness'' perspective suggested by the results of \cite{card2007}, we show that uniqueness holds for almost all initial sets satisfying an interior ball condition. This generalizes the uniqueness results seen in star-shaped or strictly expanding configurations \cite{contact}.


Unlike the case where the interior operator is the Laplacian, standard sup- and inf-convolution techniques cannot be applied directly in our setting. Due to the complexity of the operators involved, the sup-convolution of a subsolution is generally no longer a subsolution. To overcome this, we shift our focus from functions to sets, adapting the approach introduced by the authors in \cite{turanova2025hele} for free boundary problems in one spatial dimension. This strategy is effective because comparing sets rather than functions effectively restricts the comparison to a smaller class of candidates. Indeed, while Lemma \ref{L.def} establishes that viscosity subflows (respectively, superflows) correspond to subsolutions (respectively, supersolutions), the converse remains unclear.

Although our overall strategy is inspired by \cite{turanova2025hele}, generalizing that approach to higher dimensions presents significant challenges. In one dimension, the free boundary is simply given by several functions of time and non-degeneracy follows immediately from the Hopf lemma, but in higher dimensions, the geometry is much more complicated and non-degeneracy requires the free boundary to be $C^1$ in space. We address this by combining the set convolution techniques of \cite{turanova2025hele} with a set interpolation result of \cite{card2007}. However,  because the method in \cite{card2007} relies on translation-invariant operators to shift sets for testing, it cannot be applied directly to our spatially dependent operators.

The regularization of sets employed in \cite{turanova2025hele} is sufficiently flexible to accommodate the errors introduced by spatial shifts. Moreover, by imposing a Lipschitz assumption on the coefficients, we are able to handle negative free boundary speeds. In contrast, \cite{turanova2025hele} requires positive free boundary speeds with H\"{o}lder-continuous coefficients.

We will discuss in more detail the  relationship between our results and those in \cite{turanova2025hele, contact, card2007} in Remarks \ref{rem:our paper}, \ref{rem:kim intro} and \ref{rem:cardaliaguet} below.


\subsection{Notation} \label{ss:results}

We use $B_r((x,t))\subseteq\R^{d+1}$ to denote the space-time ball centered at $(x,t)\in\R^{d+1}$ with radius $r>0$. We use $\calB_r(x)\subseteq\R^{d}$ to denote the spatial ball centered at $x\in\R^{d}$ with radius $r>0$. For a set $A\in \R^{d+1}$, we write
\[
B_r(A):=\{ (x+y,t+s)\,:\, (x,t)\in A\text{ and }(y,s)\in B_r((0,0))\}.
\]
For a space-time set $\Omega\subseteq \R^{d+1}$, we denote the time slice of $\Omega$ at time $t$ by   $\Omega(t)$:
\[
\Omega(t) = \{x \,:\, (x,t)\in \Omega\}.
\]
For any set $K\subseteq\mathbb{R}^{d+1}$, its closure is denoted by $\overline{K}$, and
the complement of $K$ is denoted as $K^c$. 
For $(x,t),(y,s)\in\R^{d+1}$, we use the following notation for distance:
\[
|(x,t),(y,s)|:=|(x,t)-(y,s)|=\sqrt{|x-y|^2+|t-s|^2}.
\]
The minimal distance between $\Omega_1, \Omega_2\subseteq\R^{d+1}$ is defined as
\[
d(\Omega_1, \Omega_2)
= \inf_{(x,t)\in \Omega_1,\; (y,s)\in \Omega_2}
\, |(x,t) - (y,s)|.
\]

For a nonnegative lower semicontinuous function $p:\bbR^d\times [0,\infty)\to [0,\infty)$, we denote its positive set and the  boundary of its positive set by,
\[
\Omega_p:=\{(x,t)\in\bbR^d\times [0,\infty) \,:\, p(x,t)>0\}\quad\text{and}\quad\Gamma_p:=\{(x,t)\in\bbR^d\times [0,\infty) \,:\, x\in\Gamma_p(t)\}.
\]

Finally, we  denote the set of symmetric $d\times d$ matrices by  $\mathbb{M}^d$.

\subsection{Assumptions}
We begin by stating basic assumptions on $F$ and $V$.

\begin{assumption}[Basic assumptions]
\label{assump2} 
We assume that the following hold for
\[
F: \R^d\times [0,\infty)\times \R^d\times \mathbb{M}^d\rightarrow \R \quad \text{and}\quad V:\R^d\times \R^d\rightarrow \R.
\]
\begin{enumerate}[(i)]
    \item \label{item:assump1.1.1} $F$ is uniformly elliptic and Lipschitz-continuous in all its variables 
    and, for any $U\subset \R^d$ with $C^{1,1}$ boundary, there exists a unique viscosity, non-trivial solution of 
 \beq\lb{elliptic}
F(x,p, \nabla p,D^2p)=0\quad \text{ in }U
\eeq
with $0$ boundary data. 
   \item \label{item:assump1.1 2}
$V$ is continuous in both variables.  There exists a constant $C>0$ such that for all $\gamma\in (0,\gamma_0)$  and for all $x_1,x_2,q_1\in\bbR^d$ with $C|x_1-x_2|\leq \gamma$  we have,
\begin{align}
\lb{cond5'}
V(x_1,(1+\gamma)q_1)&\geq V(x_1,q_1);
\\
\lb{cond5}
(1+\gamma)^2 V(x_1,q_1)+C|x_1-x_2|^2/\gamma+C\gamma&\geq V(x_2,(1+\gamma)q_1).
\end{align}
\end{enumerate}
\end{assumption}

Before stating the rest of our assumptions, we introduce a definition that we'll use throughout.
\begin{definition}[Associated function]\lb{Ass}
 For each open set $U\subseteq\R^d$ with $C^{1,1}$ boundary, we define its \emph{associated function} $p(x;U)$ to be the unique, non-trivial viscosity solution of \eqref{elliptic} 
with $0$ boundary data. 
\end{definition}

We proceed with:

\begin{assumption}[Further assumptions on $F$]
\label{assump3} 
We make the following additional assumptions on $F$.
 \begin{enumerate}[(i)]

\item \label{item:assump1.1 4} Let $\Omega\subseteq \R^{d+1}$ have  $C^{1,1}$ boundary. For each $t$, define  $p(x,t)=p(x,\Omega(t))$. Then $p>0$ in $\Omega$ and $\partial_t p$, $\nabla p$ exist on $\partial\Omega$ along non-tangential directions and are continuous.


\item \label{item:assump1.1 3} Let $r,\delta\in (0,1]$, and let $ U\subseteq \R^d$ be a bounded open set with $C^{1,1}$ boundary. There exists $C>0$ that depends only on $F$ and the diameter of $U$ such that if $Cr\leq \delta$ and $\nu\in \calB_r(0)\subseteq\R^d$, then  the following holds for $u_1$ solving \eqref{elliptic} in $U$, and $u_2$ solving
\[
F(x+\nu,p, \nabla p,D^2p)=0\quad\text{ in $U$},
\]
both with zero boundary condition: for any $x_0\in\partial U$, 
we have $(1+\delta)|\nabla u_2(x_0)|\geq |\nabla u_1(x_0)|$.

\end{enumerate}
\end{assumption}

\begin{remark}
    \label{rem:assumption lemmas}
We note that the condition in Assumption \ref{assump2}\ref{item:assump1.1 2} on $V$ is weaker than those assumed in \cite{turanova2025hele}; crucially, our framework allows $V$ to take negative values.  The monotonicity condition \eqref{cond5'} is essential in the proof of the comparison principle. In Lemma \ref{lem:exampleV}, we verify Assumption \ref{assump2}\ref{item:assump1.1 2} for a particular class of velocities $V$.

Assumption \ref{assump3}\ref{item:assump1.1 3} is a technical condition which requires a certain stability of the operator. Lemmas \ref{lem:exampleF} and \ref{L.2.13}  and Remark \ref{rem:more Fs}
 provide examples of $F$'s that satisfy Assumption  \ref{assump3}. 
 We leave the question of verifying Assumption \ref{assump3}\ref{item:assump1.1 3} for a wider class of  equations for future work.  
\end{remark}



\subsection{Literature review}
\label{ss:lit}

The mathematical study of Hele-Shaw  flows has  been an active area  for the past four decades; we do not attempt to present an exhaustive list of works on this topic, but rather focus on the ones most relevant to the present work. 
Elliott and Janovsk{\'y} established the first  well-posedness result for the Hele-Shaw problem by proving the existence and uniqueness of weak solutions in $H^1$ using a variational inequality approach \cite{EJ}. For classical solutions in multidimensional settings, Escher and Simonett obtained key short-time existence results \cite{escher1997classical}. A transformative shift occurred with Kim’s introduction of viscosity solutions to the Hele-Shaw and Stefan problems, providing a robust framework for handling interfaces that may develop singularities \cite{kim2003}.  Building on related geometric front propagation theories, such as those by Barles and Souganidis \cite{barles1998new} and by Andrews and Feldman \cite{andrews2002nonlocal}, Cardaliaguet and Rouy \cite{card2004} and Cardaliaguet and Ley \cite{card2007} introduced generalized set solutions for increasing flows. We also mention that Alazard and Koch \cite{alazard2023hele} took the approach of Baiocchi and Duvaut, transforming the problem to an obstacle problem, and proved well-posedness in a broad context. 

There is also a large body of work on the regularity of the Hele-Shaw problem and related PDEs, of which we, again, highlight just a few. Several works by Kim \cite{kim2} and Choi, Jerison, and Kim \cite{CJK, choi2009local} demonstrated that if the initial surface for the one-phase Hele-Shaw problem is Lipschitz, then the free boundary becomes instantly smooth. For the Muskat problem, Dong, Gancedo, and Nguyen established global well-posedness and regularity in two spatial dimensions \cite{Hongjie} and three spatial dimensions \cite{dong23}. Schwab, Tu, and Turanova extended the well-posedness theory for viscosity solutions to the one-phase Muskat problem across all dimensions \cite{schwab2024well}. Liang \cite{liang2026} obtained global well-posedness of Lipschitz free boundaries for the point injection problem in space dimension two. Further advancements include the work of Figalli, Ros-Oton, and Serra, who established generic regularity for free boundaries by leveraging the connection to the obstacle problem \cite{figalli2020generic}. Kim and Zhang \cite{kim2024regularity} and Zhang \cite{zhang2026} proved that flat free boundaries are non-degenerate and $C^1$, even in the presence of source and drift terms.

A major theme in research concerning \eqref{main} is the behavior of these flows in complex or heterogeneous media. Kim pioneered the study of homogenization in the context of contact angle dynamics \cite{contact} and, with Mellet, investigated Hele-Shaw problems in periodic and random media \cite{KimMellet2009}. Po{\v{z}}ar further extended these results to characterize the long-time behavior of flows in random media \cite{povzar2011long} and established homogenization results for spatiotemporal periodic media \cite{povzar2015homogenization}. Recently,  the authors studied a Hele-Shaw problem with both interior and free boundary oscillations, establishing well-posedness and stochastic homogenization in one spatial dimension \cite{turanova2025hele}.

The motivation for the general form of our main equation is deeply rooted in biological modeling, particularly the spread of cancer tumors in heterogeneous environments. The problem is closely related to the inhomogeneous porous medium equation (PME) where growth is constrained by local pressure \cite{SulakTuranova}. Perthame, Quir{\'o}s, and V{\'a}zquez established the Hele-Shaw asymptotics for mechanical models of tumor growth \cite{PQV}, a direction recently expanded by Sulak and Turanova to include inhomogeneous models and their incompressible limits \cite{SulakTuranova}. Further refinements in tumor modeling, including nutrient interactions and convective effects, have been studied by Jacobs, Kim, and Tong \cite{jacobs2022tumor} and David and Schmidtchen \cite{David_S}. Beyond biology, Craig, Kim, and Yao utilized these dynamics to model congested aggregation via Newtonian interactions, framing the evolution as a gradient flow \cite{CKY}. Other macroscopic models for crowd motion have similarly employed gradient flows  to describe congested population dynamics \cite{maury2010}.


\begin{remark}
    \label{rem:our paper}
    The present paper   generalizes the well-posedness results in the authors' previous work \cite{turanova2025hele} in two significant ways. First, \cite{turanova2025hele} concerns only one spatial dimension. Second, the free boundary velocity  in  \cite{turanova2025hele} is required to be positive. 
\end{remark}

\begin{remark}
    \label{rem:kim intro}
    One of the results in Kim's work \cite{contact} also proves well-posedness for Hele-Shaw problems with velocities that can take negative values. This remark is devoted to a summary of the differences between  \cite{contact} and the present work. 

    The well-posedness results in \cite{contact} concern \eqref{main}, but for $F(x,p, \nabla p, D^2p)$ replaced by $-\Delta p$ 
   and for slightly different hypotheses on $V(x,q)$. 
In particular, the hypothesis in \cite[Section 1]{contact} require the free boundary velocity $V$ to be Lipschitz continuous, monotone in $|q|$,  and such that, for all $q\in \R^d$, $V(x,q)-|q|$ is uniformly negative.

Another key difference between \cite{contact} and our work is that \cite[Theorem 1.8]{contact} requires one of the following two conditions to hold in order to establish uniqueness: the first is that the initial set is star-shaped, and the second is a condition on the slope of the solution $p$ at the free boundary at the initial time (this is analogous to the condition in our Theorem \ref{T.3.4}). We improve upon this by establishing, in Theorem \ref{thm:generic},  that there exists  a unique viscosity flow starting from  ``most" initial sets.
\end{remark}

\begin{remark}
    \label{rem:cardaliaguet}
 In this remark we discuss the similarities and differences between our work and Cardaliaguet and Ley's \cite{card2007}, which   studies flows 
 with evolution law being the sum of a curvature term and the classical (homogeneous) Hele-Shaw term.

Our notion of viscosity flows is conceptually close to that of generalized set solutions of \cite{card2007} because both   definitions and  are based on testing with smooth-enough sets from the inside and the outside. Nevertheless, our framework remains closely linked to the  definition of viscosity solutions for the Hele-Shaw problem. For instance, to check whether a given set satisfies our definition of, say, subflow at some $(x,t)$, the test set needs to touch the candidate set only up to time $t$. This is analogous to checking whether a function satisfies the definition of a viscosity, say, subsolution of the Hele-Shaw problem  or even of a general parabolic equation \cite{usersguide}. On the other hand, \cite{card2007} requires the test set to touch the candidate set globally in time. 

An important similarity between our methods and those of \cite{cardaliaguet2000front, card2004, card2007} is in the use of a variant of Ilmanen's interposition lemma \cite{Ilmanen}, which, for two given disjoint sets, provides a $C^{1,1}$ set that touches one of them from the outside and such that its translate touches the other  set from the inside (see Theorem \ref{P.3.9}). 

A significant difference between our proof and that of \cite{card2007} lies in how we perturb the set solutions. The regularized flows that we define correspond to  classical inf- and sup-convolution of functions, as described in Remark \ref{inf sup convolution}. This technique yields a strict advancement for the movement of sets --- in other words, if $\Omega$ is a subflow (resp. superflow), the boundary velocity of its regularization is strictly greater (resp. smaller) than that of $\Omega$ itself. The precise version of this statement can be found in Lemma \ref{L.3.6}, which is a key ingredient in our proof of the comparison principle. 
In contrast, the method of \cite{card2007} requires the construction a sequence of norms in $\mathbb{R}^{d+1}$ to establish a comparison principle  (see \cite[Section 3.2 and Lemma 9]{card2007}). Consequently, our approach is more direct and accommodates more general equations with $x$-dependence under certain conditions. 

\end{remark}


\subsection{Outline}
Section \ref{ss:viscosity} is devoted to the definitions of and lemmas about viscosity solutions and viscosity flows. In Section \ref{ss:comparison}, we prove the comparison principle for viscosity flows. In Section \ref{sec:existence and uniqueness}, we use the comparison principle to establish our well-posedness results.

\subsection{Acknowledgements}
 O. Turanova acknowledges support from NSF grant DMS-2204722. Y. P. Zhang acknowledges support from NSF CAREER grant DMS-2440215 and Simons Foundation Travel Support MPS-TSM-00007305. 

\section{Viscosity flows and viscosity solutions}\label{ss:viscosity}
We begin by introducing the definitions of viscosity flows (Section \ref{s.2.2}). Then we introduce  viscosity solutions and demonstrate the relationship between visocisty flows and viscosity solutions (Section \ref{ss:visc soln def}). In Section \ref{ss:regular property}
we state a known result on interposition of sets that we'll use in the proof of Theorem \ref{L.cp}. Finally, in Section \ref{S.ex} we provide several examples of $V$ and $F$ that satisfy our assumptions.

\subsection{Viscosity flows and viscosity solutions}\lb{s.2.2}

Departing from the method in \cite{kim2003}, our framework defines viscosity flows as the primary objects that generate viscosity solutions. Consequently, our definitions characterize space-time sets rather than functions. 
Thus, we begin with a number of relevant definitions. For a space-time set $\Omega\subseteq \R^{d+1}$, we define
\[
\Omega_\#(t):=\liminf_{s\to t}\Omega(s)=\bigcup_{c\to 0}\bigcap_{|s-t|<c}\Omega(s),\qquad
\Omega^\#(t):=\limsup_{s\to t}\Omega(s)=\bigcap_{c\to 0}\bigcup_{|s-t|<c}\Omega(s).
\]

\begin{definition}[Semicontinuity]
\begin{enumerate}[(i)]
\item We say that $\Omega$ is \emph{lower semicontinuous} if
\[
\Omega(t)\subseteq \Omega_\#(t) \qquad\text{   for each $t$}.
\]

\item We say that $\Omega$ is \emph{upper semicontinuous} if
\[
\Omega(t)\supseteq\Omega(t)^\#\qquad\text{   for each $t$}.
\]
\end{enumerate}
\end{definition}
In particular, for any $\Omega\subset \R^{d+1}$, we have that $\Omega_\#$ is lower semicontinuous and  $\Omega^\#$ is upper semicontinuous.

\begin{definition}[Regular sets]
\label{def:regset}
    \begin{enumerate}[(i)]
\item We call $\Phi\subseteq \R^{d+1}$ a \emph{regular set} if it is  open  and has $C^{1,1}$ boundary. For any $(x,t)\in\partial \Phi$, we use
\[
\nu^\Phi_{x,t} =(\nu_x^\Phi,\nu_t^\Phi)\in\bbS^{d}
\]
to denote the outward unit normal to $\Phi$ at $(x,t)$.


\item Let $\Psi$ be a regular set. If $\nu_x^\Phi\neq 0$, the \emph{normal velocity} $\calK^{\Phi}_{x,t}$ at $(x,t)\in\partial \Phi$ is defined by
\[
\calK^{\Phi}_{x,t} = -{\nu_t^\Phi}/{|\nu_x^\Phi|}.
\]
If $\nu_x=0$ and $\nu_t^\Phi>0$, we define $\calK^{\Phi}_{x,t}=-\infty$. And if $\nu_x=0$ and $\nu_t^\Phi<0$, we define $\calK^{\Phi}_{x,t}=\infty$.

\item Let $\Omega\subseteq\R^d\times [0,\infty)$ and let $t>0$. A regular set $\Phi$ is said to be \emph{externally tangent} to $\Omega$ at $(x,t)\in\partial{\Omega}$ if,  for some $0<t'<t$,
\[
\Omega(s)\subseteq \Phi(s)\text{ for all }s\in [t', t] \quad\text{and}\quad (x,t)\in\partial \Phi .
\]
 A regular set $\Phi$ is said to be \emph{internally tangent} to $\Omega$ at $(x,t)\in \partial \Omega$ if,   for some $0<t'<t$,
\[
\Phi(s)\subseteq \Omega(s)\text{ for all }s\in [t', t] \quad\text{and}\quad (x,t)\in\partial\Phi.
\]


\end{enumerate}
\end{definition}

\begin{definition}[Viscosity flows]
\lb{D.1.2}
Let $T>0$.
\begin{enumerate}[(i)]
\item \label{d:subf} A space-time upper semicontinuous set $\Omega $ is said to be a \emph{viscosity subflow} of \eqref{main} in $\R^d\times(0,T)$ if the following holds. 
For every regular set 
$\Phi\subseteq \bbR^{d+1}$, if $\Phi$ is externally tangent to $\Omega$  
        at $(x_0,t_0)\in \partial\Omega$ and $\nu_x^\Phi\neq 0$, then  we have
\[
 V(x_0,\nabla p(x_0;\Phi(t_0)))\geq \calK^{\Phi}_{x_0,t_0}.
\]

\item \label{d:supf} A space-time lower semicontinuous set $\Omega $ is said to be a \emph{viscosity superflow} of \eqref{main} in $\R^d\times(0,T)$ if the following holds. 
For every regular set 
$\Phi\subseteq \bbR^{d+1}$, if $\Phi$ is internally tangent to $\Omega$  
at $(x_0,t_0)\in \partial\Omega$ and $\nu_x^\Phi\neq 0$, then  we have
\[
 V(x_0,\nabla p(x_0;\Phi(t_0)))\leq \calK^{\Phi}_{x_0,t_0}.
\]



\item We say that a space-time set $\Omega$ is a \emph{viscosity flow} of \eqref{main} if it is a viscosity subflow and  $\Omega_\#$ a viscosity superflow of \eqref{main}.
\end{enumerate}
\end{definition}
    Here the values of $\nabla p(x;\Phi(t))$ are achieved in the sense of Assumption \ref{assump3} \ref{item:assump1.1 4}.
\begin{remark}[Regularity of test sets]
    \label{rem:after visc def}
Instead of requiring the test set $\Phi$ to be $C^{1,1}$ in space-time, we could instead only require $\Phi$ to have $C^1$ space-time boundary and satisfy $\partial\Phi(t_0)\in C^{1,1}$. Indeed, for such $\Phi$, the terms in the inequalities in Definition \ref{D.1.2} are well-defined. Moreover, since any $\partial\Phi\in C^1$ with $\partial\Phi(t_0)\in C^{1,1}$ can be approximated by $C^{1,1}$ space-time sets, Definition \ref{D.1.2} is equivalent to the definition obtained by requiring only $\partial\Phi\in C^1$ and $\partial\Phi(t_0)\in C^{1,1}$.
\end{remark}



\subsection{Viscosity solutions}
\label{ss:visc soln def}
We state the notion of viscosity solution to the generic problem \eqref{main}.

\begin{definition}[Viscosity solution]\lb{D.1.2'}
Let $D\subset\R^d$ be open and let $T>0$.
\begin{enumerate}
    \item A non-negative upper semicontinuous function $p$ defined in $D\times(0,T)$ is a \emph{viscosity subsolution} of \eqref{main} on $D\times(0,T)$ if for every $\phi\in C^{2,1}_{x,t}(D\times (0,T))$ such that $p-\phi$ has a local maximum in $\left(\overline{\Omega_p}\cap\{t\leq t_0\}\right)\cap\left(D\times (0,T)\right)$ at $(x_0,t_0)$, we have:
\begin{itemize}
    \item If $p(x_0,t_0)>0$, then $
    F(x_0,p(x_0,t_0),\nabla\phi(x_0,t_0), D^2\phi(x_0,t_0))\leq 0$ holds.
    \item If $p(x_0,t_0)=0$, then either 
    \begin{align*}
    F(x_0,p(x_0,t_0),\nabla\phi(x_0,t_0), D^2\phi(x_0,t_0))&\leq 0, \quad\text{or}\\
\phi_t(x_0,t_0)-V(x_0,{\nabla\phi}(x_0,t_0))|{\nabla\phi}(x_0,t_0)|&\leq 0.
        \end{align*}
\end{itemize}

\item  
A non-negative lower semicontinuous function $p$ defined in $D\times (0,T)$ is a {\it viscosity supersolution} of \eqref{main} on $D\times(0,T)$ if for every $\phi\in C^{2,1}_{x,t}(D\times (0,T))$ such that $p-\phi$ has a local minimum in $D\times\{t\leq t_0\}$ at $(x_0,t_0)$, we have,
\begin{itemize}
    \item If $p(x_0,t_0)>0$, then $
    F(x_0,p(x_0,t_0),\nabla\phi(x_0,t_0), D^2\phi(x_0,t_0))\geq 0$ holds.
    \item If $p(x_0,t_0)=0$, then either 
    \begin{align*}
F(x_0,p(x_0,t_0),\nabla\phi(x_0,t_0), D^2\phi(x_0,t_0))&\geq 0, \quad \text{or}\\
{\nabla\phi}(x_0,t_0)&= 0, \quad\text{or}\\
\phi_t(x_0,t_0)-V(x_0,{\nabla\phi}(x_0,t_0))|{\nabla\phi}(x_0,t_0)|&\geq 0.
        \end{align*}
\end{itemize}
\end{enumerate}

\end{definition}







For a function $p(x,t)$,
define the half-relaxed limits,
\[
p^\# (x,t):=\limsup_{y\to x,\,s\to t}p(y,s)\quad\text{and}\quad p_{\# }(x,t):=\liminf_{y\to x,\,s\to t}p(y,s).
\]

\begin{definition}[Viscosity solution]\lb{def23}
We say that a non-negative function $p$ is a \emph{viscosity solution} of \eqref{main} if $p^\# $ is a viscosity subsolution of \eqref{main} and $p=p_\# $ is a viscosity supersolution of \eqref{main}.
\end{definition}

We need the following definition to relate   viscosity solutions and  viscosity flows.

\begin{definition}[Associated sub- and super-functions]
Let $T>0$.
    \begin{enumerate}
        \item Let $\Omega$ be a viscosity subflow  on $\R^d\times(0,T)$. Its \emph{associated sub-function} $p$  is defined, for each $t\in (0, T)$, as
        \[
        p(x, t) = \sup\left\{ u(x) \, :\, 
        \begin{aligned}& \text{$u$ is upper semicontinuous and is a subsolution of \eqref{elliptic}}\\&\text{on $\Omega(t)$ and $u(x)=0$ on $\overline{\Omega(t)}^c$}\end{aligned}\right\}.
        \]
         \item Let $\Omega$ be a viscosity superflow  on $\R^d\times(0,T)$. Its \emph{associated super-function} $p$  is defined, for each $t\in (0, T)$, as
        \[
        p(x, t) = \inf\left\{  u(x)  \, :\, \begin{aligned}& \text{$u$ is lower semicontinuous  and is a supersolution of \eqref{elliptic}}\\&\text{on $\Omega(t)$ and $u(x)>0$ for all $x$ in the interior of $\Omega(t)$}\end{aligned}\right\}.
        \]
    \end{enumerate}
\end{definition}


If $\Omega$ has a $C^{1,1}$ boundary, then the associated sub- and super-functions agree with the associated function.

In the next lemma, we show that if we are able to find a viscosity subflow (resp. superflow), then it yields a viscosity subsolution (resp. supersolution).

\begin{lemma}[Viscosity solutions and viscosity flows]\lb{L.def}
Suppose Assumptions \ref{assump2} and \ref{assump3} hold. If $\Omega$ is a viscosity subflow (resp. superflow) on $\R^d\times(0,T)$, then its associated sub-function (resp. super-function) $p$ is a viscosity subsolution  (resp. supersolution) on $\R^d\times(0,T)$.
\end{lemma}
\begin{proof}
We only prove the first claim. The second claim follows similarly, and is actually simpler as it requires only the first step of the following proof.

Let $\Omega$ be a viscosity subflow. Since $\Omega$ is upper semicontinuous, $p(x,t)$ is upper semicontinuous. 
Let $\phi\in C^{2,1}_{x,t}(D\times (0,T))$ be such that $p-\phi$ has a local maximum in $\left(\overline{\Omega_p}\cap\{t\leq t_0\}\right)\cap\left(D\times (0,T)\right)$ at $(x_0,t_0)$. Since $p(\cdot,t)$ is a solution  to $F=0$ in its positive set, we only need to consider the case when $p(x_0,t_0)=0$. Hence, due to Assumption \ref{assump3}\ref{item:assump1.1 4}, we have $(x_0,t_0)\in\partial\Omega$.  Without loss of generality, we can assume $\phi(x_0,t_0)=0$.

\medskip
\noindent {\bf Step 1.} In this step, we assume ${\nabla\phi}(x_0,t_0)\neq 0$. It then suffices to show 
\[
\phi_t(x_0,t_0)-V(x_0,{\nabla\phi}(x_0,t_0))|{\nabla\phi}(x_0,t_0)|\leq 0,
\]
whenever the following holds
\[
    F(x_0,0,\nabla\phi(x_0,t_0), D^2\phi(x_0,t_0))> 0\quad\text{and}\quad {\nabla\phi}(x_0,t_0)\neq 0.
\]
Since ${\nabla\phi}(x_0,t_0)\neq 0$, the level set $\phi=0$ is $C^{2,1}_{x,t}$ locally near $(x_0,t_0)$.

By continuity  of the operator $F$, we can modify $\phi$  outside a neighborhood of $(x_0,t_0)$ such that
\[
F(x,\phi(x,t),\nabla\phi(x,t), D^2\phi(x,t))\geq 0 \quad \text{  for all $(x,t)$},
\]
the level set $\phi=0$ is $C^{2,1}_{x,t}$ globally, and $\Phi:=\{\phi>0\}$ contains  $\{p>0\}\cap \{t\leq t_0\}$.  It then suffices to show
$\calK^{\Phi}_{x_0,t_0}\leq V(x_0,{\nabla\phi}(x_0,t_0))$.

Now, the comparison principle for $F$ in the domain $\Phi(t_0)$ yields
$\phi(x,t_0)\geq p(x;\Phi(t_0))$.
Since both $\nabla \phi(x_0,t_0)$ and $\nabla p(x_0;\Phi(t_0))$ are pointing in the outer normal direction at $x_0\in\partial\Phi(t_0)$, we have
\[
|\nabla p(x_0;\Phi(t_0))|\leq |\nabla\phi(x_0,t_0)| .
\]
Hence, the monotonicity of $V$ --- see assumption \eqref{cond5'}  ---  yields
\[
V(x_0,\nabla p(x_0;\Phi(t_0)))\leq V(x_0,{\nabla\phi}(x_0,t_0)).
\]

Furthermore, we note that $\Phi$ is externally tangent to $\Omega$ at $(x_0,t_0)$. Since, 
 by  assumption,  $\Omega$ is a viscosity subflow, we get
\[
\calK^{\Phi}_{x_0,t_0}\leq V(x_0,\nabla p(x_0;\Phi(t_0))). 
\]
Combining the two previous inequalities  yields $\calK^{\Phi}_{x_0,t_0}\leq V(x_0,{\nabla\phi}(x_0,t_0))$,  finishing the proof.

\medskip
\noindent {\bf Step 2.} For the general case, for any $\eps\in (0,1)$, let us consider 
\[
\Omega_\eps:=\bigcup_{t\in [0,T)} (\calB_\eps(0)+\Omega(t))\times\{t\}.
\]
We claim that $\Omega_\eps$ is a viscosity subflow of 
\begin{equation}\lb{maineps}
\left\{
\begin{aligned}
    &F_\eps(x,p,\nabla p,D^2 p)=0 &&\quad\text{ in }\{p>0\},\\
  & \partial_t p=V_\eps(x,  \nabla_x p)| \nabla_x p| &&\quad\text{ on }\partial\{p>0\}
\end{aligned}   
\right.
\end{equation}
where
\[
F_\eps(x,p,q,X):=\inf_{y\in \calB_\eps(x)}F(y,p,q,X)\quad\text{and}\quad V_\eps(x,q):=\sup_{y\in \calB_\eps(x)}V(y,q).
\]
Indeed, suppose that a regular set $\Phi_\eps$ is externally tangent to $\Omega_\eps$ at $(z_\eps, t_\eps)\in\partial\Omega_\eps$ and $\nu_x^{\Phi_\eps}\neq 0$. Let $z_0\in \partial\Omega$ be such that $|z_0-z_\eps|=\eps$. Then for $\nu_0:=z_0-z_\eps$, by the definition of $\Omega_\eps$, we have
$\Phi_0:=\Phi_\eps+(\nu_0, 0)$ is externally tangent to $\Omega$. Hence, the definition of viscosity subflow yields
\begin{align}\lb{2.10}
 V(z_0,\nabla p(z_0;\Phi_0(t_\eps)))\geq \calK^{\Phi_0}_{z_0,t_\eps}.
\end{align}
Let $\tp_\ep$ be the unique viscosity solution of $F_\ep(x, \tp_\ep, \nabla \tp_\ep, D^2p_\ep)=0$ in the domain $\Phi_\eps(t_\eps)$ with $0$ boundary value. And, let $\bar p(x;\Phi_0(t_\ep))$ be as given by Definition \ref{Ass}; i.e. $\bar p(x;\Phi_0(t_\ep))$ is the unique viscosity solution of \eqref{elliptic} in $\Phi_0(t_\ep)$ with 0 boundary value. The definition of $F_\ep$ yields that $\bar p(x+\nu_0;\Phi_0(t_\ep))$ is a subsolution of $F_\eps=0$ in the domain  $\Phi_\eps(t_\eps)$. Note that both $\nabla \tp_\eps (z_\eps)$ and $\nabla \bar p(z_0;\Phi_0(t_\eps))$ point in the outer normal direction of $\Phi_\eps(t_\eps)$ at $z_\eps$. And, the comparison principle  implies that
$\tp_\eps(x)\geq \bar p(x+\nu_0;\Phi_0)$  for $x\in\Phi_\eps(t_\eps)$.
Therefore, we find $|\nabla \tp_\eps (z_\eps)|\geq | \nabla \bar p(z_0;\Phi_0(t_\eps))|$.
This, the definition of $V_\ep$,  and the monotonicity of $V$ imply
\[
V_\eps(z_\eps,\nabla \tp_\eps(z_\eps)) \geq V(z_0,\nabla \tp_\eps(z_\eps))\geq V(z_0,\nabla \bar p(z_0;\Phi_0(t_\eps))).
\]
Since $\Phi_0=\Phi_\eps+(\nu_0, 0)$, we have $\calK^{\Phi_0}_{z_0,t_\eps}=\calK^{\Phi_\eps}_{z_\eps,t_\eps}$. Thus, we conclude from the previous line and \eqref{2.10} that 
\[
V_\eps(z_\eps,\nabla \tp_\eps(z_\eps))\geq \calK^{\Phi_\eps}_{z_\eps,t_\eps}
\]
which proves the claim.

Now, let $p_\ep$ denote the associated sub-function of $\Omega_\ep$. We shall use Step 1 to conclude that $p_\eps$ is a viscosity subsolution to \eqref{maineps}. Indeed, since $\Omega_\eps(t)$ satisfies the interior ball condition for each $t\in (0,T)$,  the Hopf maximum principle implies that  $p_\ep(\cdot, t)$ is non-degenerate near the boundary of  $\Omega_\eps(t)$. Consequently, if $\phi\in C^{2,1}_{x,t}(D\times (0,T))$ is such that $p_\ep-\phi$ has a local maximum in $\left(\overline{\Omega_p}\cap\{t\leq t_0\}\right)\cap\left(D\times (0,T)\right)$ at some $(x_0,t_0)\in \partial \Omega_\ep$, the condition $\nabla\phi(x_0,t_0)\neq 0$ must hold, and thus the arguments in Step 1 apply.

Finally, let $p$ be the associated sub-function  of $\Omega$. If we let $\tilde p(x,t)$ be the limit of $p_\eps$ as $\eps\to 0$, then $\tilde p(\cdot,t)=0$ on $\overline{\Omega(t)}^c$. Also, since $\tilde p(\cdot,t)$ is a subsolution to \eqref{elliptic} on $\Omega(t)$ by the convergence of $\Omega_\eps\to \Omega$ and $F_\eps\to F$, we obtain $\tilde p(\cdot,t)\leq p(\cdot,t)$. On the other hand, the definition of $\Omega_\eps$ yields $p_\eps(x,t)\geq p(x,t)$. Thus, we get $\tilde p(\cdot,t)= p(\cdot,t)$.
Moreover, since $V_\eps\to V$, $F_\eps\to F$ and $p_\eps(\cdot,t)\to p(\cdot,t)$ locally uniformly, we pass $\eps\to 0$ and apply the stability property of viscosity solutions to get that
$p$ is indeed a viscosity subsolution. More precisely, if a test function $\psi$ touches $p$ from above at $(x_0,t_0)$ for $t\leq t_0$, then it touches $p_\eps$ from above locally at some $(x_\eps,t_\eps)$ for $t\leq t_\eps$, and $(x_\eps,t_\eps)\to (x_0,t_0)$ as $\eps\to 0$. Since $p_\eps$ is a viscosity subsolution of \eqref{maineps}, then either one of the following holds:
\begin{align*}
F_\eps(x_\eps, p_\eps(x_\eps,t_\eps), \nabla \psi(x_\eps,t_\eps), D^2\psi(x_\eps,t_\eps))&\leq 0\quad\text{or}\\
\partial_t\psi_\eps(x_\eps,t_\eps)-V_\eps(x_\eps,{\nabla\psi}(x_\eps,t_\eps))|{\nabla\psi}(x_\eps,t_\eps)|&\leq 0.
\end{align*}
Passing $\eps\to 0$ yields the conclusion.

\end{proof}

\subsection{An interposition result}
\label{ss:regular property}

The following result allows for two ``nice" sets to be placed between two disjoint sets. It follows directly from applying \cite[Theorem 3.3(i)]{card2007} with $\Omega_1\cap \R^d\times [0,t_1]$ and $\Omega_2\cup \R^d\times (\R\backslash [0,t_2])$ in place of $C_1$ and $C_2$, respectively, and after accounting for our definitions of external and internal tangency, which differ slightly from those used in \cite{card2007}.

\begin{theorem}[Interposition of sets]\lb{P.3.9}
Let $\Omega_1,\Omega_2\subseteq \bbR^{d+1}$, and $(x_1,t_1)\in \overline{\Omega_1}$ and $(x_2,t_2)\in \overline{\Omega_2}$ satisfy
\begin{enumerate}
    \item $\Omega_1\subseteq \Omega_2$, $\Omega_1$ is compact, and $\Omega_2$ is open;

    \item $|(x_1,t_1)-(x_2,t_2)|=d(\Omega_1\cap \R^d\times [0,t_1],\overline{\Omega_2^c}\cap \R^d\times [0,t_2])$ and $x_1\neq x_2$.
\end{enumerate}
Then there exists a compact set $\Sigma_1\subseteq\R^{d+1} $ with $C^{1,1}$ boundary such that $\Sigma_1$ is externally tangent to $\Omega_1$ at $(x_1,t_1)$, and $\Sigma_2 := \Sigma_1 + (x_2,t_2) - (x_1,t_1)$
is internally tangent to $\Omega_2$ at $(x_2,t_2)$.
\end{theorem}

%


\subsection{Examples}\lb{S.ex}

In this section we discuss several typical examples of $V$ and $F$ that satisfy Assumption \ref{assump2}.

\begin{lemma}[Example velocity $V$]
\label{lem:exampleV}
    Let $a(x)\geq c>0$ and $b(x)$ be uniformly bounded and Lipschitz continuous. 
    There exists $C>0$, depending only on $c$, $\|a\|_{\Lip}$, and $\|b\|_{\infty}+\|b\|_{\Lip}$, such that $V$ defined by  $V(x,q)=a(x)|q|+b(x)$ satisfies Assumption \ref{assump2}\ref{item:assump1.1 2}.
\end{lemma}

\begin{proof}
It suffices to verify \eqref{cond5}.
For any $x_1, x_2$ such that $C|x_1-x_2|\leq \gamma$, we have
    \begin{align*}
&(1+\gamma)^2 V(x_1,q)+C|x_1-x_2|^2/\gamma+C\gamma-V(x_2,(1+\gamma)q)\\
& = (1+\gamma)^2a(x_1)|q|+(1+\gamma)^2b(x_1) +C|x_1-x_2|^2/\gamma+C\gamma - a(x_2)(1+\gamma)|q|-b(x_2)\\
& = |q|\gamma a(x_1) +|q|(1+\gamma)(a(x_1)-a(x_2))+ \gamma^2|q|a(x_1) +b(x_1)-b(x_2)+ (2\gamma +\gamma^2) b(x_1)+ \\
&\quad \quad +C|x_1-x_2|^2/\gamma+C\gamma\\
&\geq c\gamma|q|+C|x_1-x_2|^2/\gamma+C\gamma-3\gamma\|b\|_\infty-(2|q|\|a\|_{\Lip}+\|b\|_{\Lip})|x_1-x_2|.
\end{align*}
If $C$ is large enough, we have,
\[
C|x_1-x_2|^2/\gamma+C\gamma-3\gamma\|b\|_\infty-(\|b\|_{\Lip})|x_1-x_2|\geq 0,
\]
and, since $C|x_1-x_2|\leq \gamma$,
\[
c\gamma|q|-2|q|\|a\|_{\Lip}|x_1-x_2|\geq c\gamma|q|-2|q|\|a\|_{\Lip}\gamma/C\geq 0.
\]
Combining the two inequalities yields \eqref{cond5}.
\end{proof}

Assumption \ref{assump3}\ref{item:assump1.1 3} is a certain continuity dependence of coefficients property of the elliptic operator $F$. It is not easy to verify this assumption in general; in the next two lemmas, we provide two examples of operators for which Assumption \ref{assump3}\ref{item:assump1.1 3} holds. 

\begin{lemma}[Example of $F$]\lb{L.2.12}
\label{lem:exampleF}
Let $F(x,z,q,X):=-\tr (X)+v\cdot q+b(x)z-a(x)$, where $v\in\R^d$, $a(x)$ and $b(x)$ are uniformly bounded and Lipschitz continuous, and $a(x)\geq c>0$. Then $F(x,z,q,X)$ satisfies Assumption \ref{assump3}\ref{item:assump1.1 4}\ref{item:assump1.1 3}.  
\end{lemma}

\begin{proof}
Assumption \ref{assump3}\ref{item:assump1.1 4} follows from the classical regularity theory for linear elliptic equations (see e.g., Theorem 4.2 in Chapter 5 of \cite{elliptic}). 


Now, we show that $u:=(1+\delta)u_1$ is a supersolution to $F(x+\nu,u, \nabla u,D^2u)=0$. Once this is done, the comparison principle yields that $(1+\delta)u_1\geq u_2$ in $U$. This concludes the proof that $F$ satisfies Assumption \ref{assump3}\ref{item:assump1.1 3}.

By \cite[Chapter 5, Theorem 1.1]{elliptic}, we know that $\sup_{U}|u_1|\leq C$ for some $C$ depending only on $d$, the coefficients, and the diameter of $U$.
Direct computation yields
\begin{align*}
F(x+\nu,u, \nabla u,D^2u)&=(1+\delta)(-\Delta u_1(x)+v\cdot\nabla u_1(x)+b(x+\nu)u_1(x))-a(x+\nu) .
\end{align*}
Since $-\Delta u_1(x)+v\cdot\nabla u_1(x)+b(x)u_1(x)-a(x)=0$, using the assumptions, we get
\[
F(x+\nu,u, \nabla u,D^2u)\geq c\delta -|\nu|((1+\delta)\|u_1\|_\infty\|b\|_{\Lip}+\|a\|_{\Lip})
\]
which is non-negative if $\delta\geq Cr$ with $C$ depending only on $d$, the coefficients, and the diameter of $U$.
\end{proof}

\begin{remark}
    \label{rem:more Fs}
We remark that operators of the following form satisfy Assumption \ref{assump3}\ref{item:assump1.1 4}\ref{item:assump1.1 3} as well. Under the hypotheses of Lemma \ref{L.2.12}, let $T:\bbR^d\to \bbR^d$ be $C^2_{\rm loc}$, invertible, and such that $\|T^{-1}\|_{\rm Lip}<\infty$, and take
\beq\lb{Fy}
\widetilde F(y,z,q,X):=-\text{tr}\left(A(y) X\right) + \widetilde{v}(y) \cdot q + \tilde{b}(y)z - \tilde{a}(y) 
\eeq
where  $\tilde{b}(y) = b(T^{-1}(y))$, $\tilde{a}(y) = a(T^{-1}(y))$, and
\[
A(y) = DT(x) DT(x)^T \Big\vert{}_{x = T^{-1}(y)},\quad \widetilde{v}(y) = DT(x) v - \Delta T(x) \Big\vert{}_{x = T^{-1}(y)}.
\]
Indeed, suppose $u_1(y)$ is a solution to \eqref{Fy} and $u_2$ is a solution to $\widetilde F(y+\nu,u_2,Du_2,D^2u_2)=0$. Then $w_i(x):=u_i(T(x))$, with $i=1,2$, satisfy $F(x,w_1,Dw_1,D^2w_1)=0$ and 
$F(T^{-1}(T(x)+\nu),w_2,Dw_2,D^2w_2)=0$. By our assumption on $T$, we have
\[
|T^{-1}(T(x)+\nu)-x|= |T^{-1}(T(x)+\nu)-T^{-1}(T(x))|\leq \|T^{-1}\|_{\rm Lip}|\nu|.
\]
The proof of Lemma \ref{L.2.12} yields $(1+\delta)w_1\geq w_2$ and so $(1+\delta)u_1\geq u_2$. Hence, $\widetilde F$ satisfies Assumption \ref{assump3}\ref{item:assump1.1 3}. And, it's easy to verify that Assumption \ref{assump3}\ref{item:assump1.1 4} is satisfied as well.
\end{remark}





\begin{lemma}[Example of $F$ in one spatial dimension]\lb{L.2.13}
Suppose $d=1$ and $R>1$. Let $F(x,z,q,X):=-\tr (X)+c(x) q+b(x)z-a(x)$, where $a(x),b(x)$ and $c(x)$ are uniformly bounded in $C^1$, and $\inf_{x}a(x),\inf_{x}b(x)>0$. Then $F(x,z,q,X)$ satisfies Assumption \ref{assump3}\ref{item:assump1.1 3} as long as $U\subseteq (-R,R)$.  
\end{lemma}

\begin{proof}
Let $u=p(\cdot;U)$ be as in Definition \ref{Ass}, with $U$ an open interval in $(-R,R)$.
First of all, the comparison principle yields that $0\leq u\leq  \sup_x\frac{a(x)}{b(x)}$. Next, because $u=0$ on the two ends of the interval of $U$ and $u\geq 0$, there exists $x_0\in U$ such that $u_x(x_0)=0$. It follows from the equation that for any $x_1\in U$,
\[
u_x(x_1)=\int_{x_0}^{x_1} c(x)u_x(x)+b(x)u(x)-a(x) dx.
\]
Hence,
\[
u_x(x_1)= c(x_1)u(x_1)-c(x_0)u(x_0)+\int_{x_0}^{x_1} -c_x(x)u(x)+b(x)u(x)-a(x) dx.
\]
Since $a,b,c,c_x$ and $u$ are finite on $(-R,R)$, this implies that
$|u_x|$ is uniformly bounded in $U$ independent of $U$.

Now, we take $\nu\in\R$ and $|\nu|<r$. It suffices to show that $\tilde u:=(1+\delta)u$ is a supersolution to $F(x+\nu,u, \nabla u,D^2u)=0$ when $1>\delta\geq C'r$ for some $C'$ independent of $U$. Indeed,
\begin{align*}
F(x+\nu,\tilde u, \nabla \tilde u,D^2\tilde u)&=(1+\delta)(-u_{xx}+c(x+\nu)u_x+b(x+\nu)u)-a(x+\nu)    \\
&\geq (1+\delta)(-u_{xx}+c(x)u_x+b(x)u)-a(x)\\
&\qquad-(2\|c\|_{\Lip}\|u_x\|_\infty+2\|b\|_{\Lip}\|u\|_\infty+\|a\|_{\Lip})r\\
&\geq \delta a(x)-C(\|u_x\|_\infty+\|u\|_\infty+1)r.
\end{align*}
Note that $\inf_x a>0$ and we've already shown that $u$ is uniformly  $C^1$ independent of $U$. Hence, we proved that $\tilde u$ is a supersolution if $C'r\leq \delta<1$ for some $C'$.
\end{proof}

\section{Comparison principle for viscosity flows}
\label{ss:comparison}

In this section, we establish the comparison principle for viscosity flows  and well-posedness of viscosity solutions  of \eqref{main}.

Similarly to the setting of free boundary problems in \cite{cbook}, we need to regularize our viscosity sub- and super- flows to establish comparison. Section \ref{ss:reg flows} is devoted to  defining  the regularizations and establishing some of their important properties. In particular, Lemmas \ref{L.3.4}, \ref{L.3.3}, and \ref{L.3.6} are at the heart of the proof of the comparison principle (Theorem \ref{L.cp}) in Section \ref{ss:comp and wp}.  

\subsection{Regularized flows}
\label{ss:reg flows}

Let two bounded sets $\Omega_1$ and $\Omega_2$  be, respectively, viscosity sub- and superflows to \eqref{main} 
in $\R^d \times (0,T)$. For some $r,\delta\in (0,1)$ and $h\in [0,1)$, we rescale the set in time as follows:
\beq\lb{5.0}
\Omega_1^\delta :=\left\{(x,(1+\delta)^{2}t)\,:\,  (x,t)\in\Omega_1\right\},
\quad
\Omega_2^\delta :=\left\{(x,t)\,:\,  (x,(1+\delta)^{2}t)\in\Omega_1\right\},
\eeq
It is direct to check that $\Omega_1^\delta$ and $\Omega_2^\delta$ are, respectively, viscosity sub- and superflows to \eqref{main} with $V$ replaced by $(1+\delta)^{-2}V$ and ${(1+\delta)^{2}}V$.
Next, for  $h\in (0,r/T)$, we define $\Omega_1^{\delta,r,h}$ to be a neighborhood of $\Omega_1^\delta $,
\beq\lb{5.1}
\Omega_1^{\delta,r,h}:=\left\{(x,t)\, \big\vert \, t<r/h,\, |(x,t)-(y,s)|<r- ht,\,(y,s)\in\Omega_1^\delta \right\},
\eeq
and we define $\Omega_2^{\delta,r,h}$ to be an interior subset of $\Omega_2^\delta $,  
\beq\lb{5.2}
\Omega_2^{\delta,r,h}:=\left\{(x,t) \, \big\vert \, t<r/h,\, |(x,t)-(y,s)|>r-ht,\,(y,s)\in (\Omega_2^\delta )^c\right\}.
\eeq
Since we can rewrite
\begin{align*}
\Omega_1^{\delta,r,h}&=\bigcup_{s\in (0,T),|(\Delta x,\Delta t)|<r-h(s+\Delta t)}\left\{ (x+\Delta x,t+\Delta t)\,:\, (x,t)\in \Omega_1^\delta,\, t<s\right\}, \\
\Omega_2^{\delta,r,h}&=\bigcap_{s\in (0,T),|(\Delta x,\Delta t)|<r-h(s+\Delta t)}\left\{ (x+\Delta x,t+\Delta t)\,:\, (x,t)\in \Omega_2^\delta,\, t<s\right\},
\end{align*}
we see that $\Omega_1^{\delta,r,h}$ is upper semicontinuous and $\Omega_2^{\delta,r,h}$ is lower semicontinuous.

\begin{remark}
    \label{inf sup convolution}
The definitions of $\Omega_1^{\delta,r,h}$ and $\Omega_2^{\delta,r,h}$ are closely related to the classical inf- and sup-convolution of functions. Suppose that $\Omega_1^\delta$ is a regular set. We can let $p_1$ be the associated function of $\Omega_1^\delta$ (as in Definition \ref{Ass}) and define the sup-convolution of $p_1$ as
\[
v_1(x,t):=\sup_{|(y,s)-(x,t)|<r -ht}p_1(y,s).
\]
Then it can be verified directly that $\Omega_1^{\delta,r,h}=\{v_1>0\}$. Similarly, if $\Omega_2^\delta$ is regular and $p_2$ is its associated function, then $\Omega_2^{\delta,r,h}=\{v_2>0\}$, where
\[
v_2(x,t):=\inf_{|(y,s)-(x,t)|<r -ht}p_2(y,s).
\]
\end{remark}

Before proceeding, we introduce one more definition. 
\begin{definition}[Regular free boundary points]Let $\Omega\subset \R^d\times (0, T)$ and let $(x_0,t_0)\in\partial\Omega$ with $t_0\in (0,T)$.
\begin{enumerate}
    \item 
We say that $(x_0,t_0)$ is {\it regular from the interior} if there exists a space-time ball $B\subseteq \Omega$ such that $\overline{\Omega^c}\cap \overline B=\{(x_0,t_0)\}$. 

\item We say that $(x_0, t_0)$ is {\it regular from the exterior} if there exists a space-time ball $B\subseteq \Omega^c$ such that $\overline\Omega\cap \overline B=\{(x_0,t_0)\}$. 
\item We say that $(x_0, t_0)$ is {\it regular} if it is both regular from the interior and from the exterior.
\end{enumerate}
\end{definition}
In Lemma  \ref{L.3.1} below, we'll show that the boundary points of $\Omega_1^{\delta,r,h}$ are regular from the interior, and the boundary points of $\Omega_2^{\delta,r,h}$ are regular from the exterior.


\begin{lemma}[Regular free boundary points]\lb{L.3.1}
Let $T>0$ and $\delta, r, h\in [0,1)$, and let two bounded sets $\Omega_1$ and $\Omega_2$  be, respectively, viscosity sub- and superflows to \eqref{main} 
in $\R^d \times (0,T)$. Let $\Omega_i^\delta$ and $\Omega_i^{\delta,r,h}$, for $i=1,2$, be as defined in  \eqref{5.0}-\eqref{5.2}.

\begin{enumerate}[(i)]
\item \lb{L.3.1.i}If $(x_1,t_1)\in\partial \Omega_1^{\delta,r,h}$ and $(y_1,s_1)\in\partial\Omega_1^\delta$ satisfy
$\left|(y_1,s_1)-(x_1,t_1)\right|=r-ht_1$,
then
$(x_1,t_1)$ is regular from the interior of $\Omega_1^{\delta,r,h}$, and $(y_1,s_1)$ is regular from the exterior of $\Omega_1^\delta$.

Moreover, one unit outer normal vector of $\Omega_1^{\delta,r,h}$ at $(x_1,t_1)$ is given by $\nu_1/|\nu_1|$ where
\[
\nu_1:=\left({x_1-y_1},{t_1-s_1+h(r-ht_1)}\right).
\]

\item \lb{L.3.1.ii} If $(x_2,t_2)\in\partial \Omega_2^{\delta,r,h}$ and $(y_2,s_2)\in\partial\Omega_2^\delta$ satisfy
$\left|(y_2,s_2)-(x_2,t_2)\right|=r-ht_2$,
then
$(x_2,t_2)$ is regular from the exterior of $\Omega_2^{\delta,r,h}$, and $(y_2,s_2)$ is regular from the interior of $\Omega_1^\delta$.

Moreover, one unit outer normal vector of $\Omega_2^{\delta,r,h}$ at $(x_2,t_2)$ is given by $\nu_2/|\nu_2|$ where
\[
\nu_2:=\left({y_2-x_2},{s_2-t_2-h(r-ht_2)}\right).
\]
\end{enumerate}
\end{lemma}
\begin{proof}
We only provide the proof of \ref{L.3.1.i}, as the proof of \ref{L.3.1.ii} is very similar. Note that $(y_1,s_1)\in\partial\Omega_1^\delta$ and 
\[
(y_1,s_1)\in \left(\partial B_{r-ht_1}((x_1,t_1))\right)\bigcap (\partial\Omega_1^\delta).
\]
Thus, $(y_1,s_1)$ is regular  from exterior of $\Omega_1^\delta$.
Moreover, if  letting
\[
A_1:=\left\{(x,t)\,:\,  |(x,t)-(y_1, s_1)|\leq r-ht\right\},
\]
we have $A_1\subseteq \Omega_1^{\delta,r,h}$ and $(x_1,t_1)\in(\partial A_1)\cap (\partial \Omega_1^{\delta,r,h})$. Hence, $(x_1, t_1)$ is regular from the interior of $\Omega_1^{\delta,r,h}$.
Since $(y_1,s_1)$ is in the interior of $\Omega_1^{\delta,r,h}$, a direct computation yields that the outer normal vector of $\partial A_1$ at $(x_0,t_0)$ is parallel to
$
\left({x_1-y_1},{t_1-s_1+h(r-ht_1)}\right)$.
\end{proof}

The next lemma gives us even more information about regular points of $\Omega_1^{\delta, r, h}$ and $\Omega_1^{\delta, r, h}$. Namely, if we happen to know that $\nu_i$ is the \emph{unique} normal vector of $\Omega_i^{\delta, r, h}$ for $i=1,2$, and that $\nu_1$ and $\nu_2$ point in the same direction, then we can conclude that $x_1\neq x_2$. In particular, we find that the exterior tangent plane of $\Omega_1^{\delta,r,h}$ at $(x_1,t_1)$ and the interior tangent plane of $\Omega_2^{\delta,r,h}$ at $(x_2,t_2)$ are not perpendicular to $(0,1)$.

Notice that $x_1\neq x_2$ is a hypothesis of Theorem \ref{P.3.9} (interposition of sets); both that theorem and this lemma will play key roles in the proof of our comparison result.

\begin{lemma}[Tangent planes not perpendicular]
    \label{L.3.4}
Suppose that the hypotheses of Lemma \ref{L.3.1}\ref{L.3.1.i}\ref{L.3.1.ii}  hold. Suppose also that  $\nu_i$ is the \emph{unique} outer normal vector for $\Omega_i^{\delta, r, h}$ at $(x_i, t_i)$, for $i=1,2$, and that, moreover, 
\beq\lb{3.1}
\nu/|\nu|=\nu_1/|\nu_1|=\nu_2/|\nu_2|
\eeq
where 
\beq
\label{nu}
\nu:=(x_2-x_1,t_2-t_1)\neq 0.
\eeq
Then we have $x_1\neq x_2$. Furthermore, 
\beq
\label{Ksigma<infty 1}
|\calK^{\Sigma_1}_{x_1,t_1}|,|\calK^{\Sigma_2}_{x_2,t_2}|<\infty.
\eeq
\end{lemma}

\begin{proof}
Let us begin with the case $t_1=t_2$.  Since, according to  \eqref{nu}, we have $\nu=(x_2-x_1,t_2-t_1)\neq 0$, we find that $x_1\neq x_2$ must hold.

 Thus, we are left to consider the case $t_1\neq t_2$. We will establish
\beq\lb{3.2}
\text{if $t_1>t_2$ then $x_1\neq y_1$ and if $t_1<t_2$ then $x_2\neq y_2$},
\eeq
and this will finish the proof. Indeed, let us suppose that \eqref{3.2} holds. 
Since $t_1\neq t_2$ also holds, we have that either $x_1\neq y_1$ or $x_2\neq y_2$. Thus, the explicit expressions for $\nu_1$ and $\nu_2$ yield that the spatial component of either $\nu_1$ or $\nu_2$ is nonzero. Therefore, the equalities in \eqref{3.1}  imply that the spatial components of $\nu$, $\nu_1$, and $\nu_2$ are all nonzero, so the explicit expression \eqref{nu} for $\nu$  gives $x_1\neq x_2$.  

Therefore, it remains to prove \eqref{3.2}. We provide the proof of \eqref{3.2} only for the case when $t_1>t_2$, as the case $t_1<t_2$ follows similarly.
Let us assume, for the sake of contradiction, that $x_1=y_1$.  (Then \eqref{nu} and \eqref{3.1} imply  that $x_1=x_2$ holds as well.) Since $t_1>t_2$, the explicit expression \eqref{nu} for $\nu$ and the equalities in \eqref{3.1} imply that both $\nu$ and $\nu_1$ point in the negative $t$-direction. Consequently using the explicit expression for $\nu_1$, we find $t_1<s_1$. Note that
\[
d((x_1,t_1),\Omega_1^{\delta})=|(y_1, s_1) - (x_1, t_1)|= r-ht_1,
\]
and so $(y_1,s_1)$ is regular from the exterior of $\Omega_1^\delta$ and $B_1:=B_{r-ht_1}((x_1,t_1))$ is one exterior ball tangent to 
$\Omega_1^{\delta}$ at $(y_1,s_1)$.

Now we define the test set $\Phi:= \calB_R(y_1)\times (0,t_1+1)\backslash B_1$, where  $R$ is large enough such that $\Omega_1^\delta\cap (0,t_1)\subseteq\Phi$. (Recall that $\calB_R(y_1)\subset\R^d$ denotes the spatial ball of radius $R$ centered at $y_1$.) We have that $\Phi$ is externally tangent to $\Omega_1^\delta$ at $(y_1,s_1)$. 
Let $p(\cdot;\Phi(t))$ satisfy the elliptic equation \eqref{elliptic} in $\Phi(t)$ with $0$ boundary data (as in Definition \ref{Ass}).  Since $\Omega_1^\delta$ is a viscosity subflow to \eqref{main} with $V$ replaced by $(1+\delta)^{-2}V$, we have 
\[
(1+\delta)^{-2}V(x_0,p(y_1;\Phi(s_1)))\geq \calK^{\Phi}_{y_1,s_1}.
\]
 Assumption \ref{assump3}\ref{item:assump1.1 4}  implies that 
$\nabla p(y_1;\Phi(s_1))$ (defined via non-tangential limit) is finite, and so
\[
V(x_0,p(y_1;\Phi(s_1)))<\infty.
\]
But, since $t_1<s_1$ and $x_1=y_1$, we have that $\calK^{\Phi}_{y_1,s_1}=+\infty$, 
which, together with the two previous inequalities, yields the desired contradiction. Therefore, we conclude that $x_1\neq y_1$ holds.

Finally, recalling the notation in Definition \ref{def:regset}, we have, for $i=1,2$, 
\[
\calK^{\Sigma_i}_{x_i,t_i}= -\nu^{\Sigma_i}_{t_i}/\nu^{\Sigma_i}_{x_i}.
\]
The fact that $x_1\neq x_2$, together with \eqref{3.1} and \eqref{nu}, imply that $\nu^{\Sigma_i}_{x_i}\neq 0$ for $i=1,2$. Therefore, \eqref{Ksigma<infty 1} holds, completing the proof.
\end{proof}


As usual, suppose that $\Omega_1$ is a subflow of \eqref{main}. 
Since $\Omega_1^\delta$ is a subflow of \eqref{main}, we know that if a regular set is externally tangent to it at a point, then the inequality in Definition \ref{D.1.2}\ref{d:subf} holds (with $V$ replaced by $(1+\delta)^{-2}V$). However, in the proof of our comparison result (Theorem \ref{L.cp}) we will be in a situation where a regular set  is externally tangent to $\Omega_1^{\delta, r, h}$. Can we deduce that some perturbation of that set is externally tangent to $\Omega_1^\delta$? Part \ref{L.3.3.i} of the following lemma answers this question  in the affirmative, and, moreover, relates the normal velocities of the two sets. In part \ref{L.3.3.ii} we establish an analogous result but for superflows and their regularizations.


We remark that the statement of the following lemma is similar to that of \cite[Lemma 4.7]{card2007}. However, since the present work and \cite{card2007} construct different approximations of sub/superflows, the proof of this lemma  differs from that of \cite[Lemma 4.7]{card2007}. 

Recall that  $\calK^{\Sigma}_{x,t}$ denotes the normal velocity of a regular set $\Sigma\subseteq\R^{d+1}$ at $(x,t)\in\partial \Sigma$.


\begin{lemma}[Tangent sets to $\Omega_i^{r, \delta, h}$ and $\Omega_i^\delta$] \lb{L.3.3}
Let $T>0$ and $\delta, r, h\in [0,1)$, and let two bounded sets $\Omega_1$ and $\Omega_2$  be, respectively, viscosity sub- and superflows to \eqref{main} 
in $\R^d \times (0,T)$. Let $\Omega_i^\delta$ and $\Omega_i^{\delta,r,h}$, for $i=1,2$, be as defined in  \eqref{5.0}-\eqref{5.2}.
\begin{enumerate}[(i)]
    \item \label{L.3.3.i}
Assume that an open set $\Sigma\subseteq\R^{d+1}$ with $C^2$ boundary is externally tangent to $\Omega_{1}^{\delta,r,h}$
at some point $(x_1, t_1) \in \partial \Omega_{1}^{\delta,r,h}$. 
Let $(y_1, s_1) \in \partial\Omega_{1}^\delta$ be such that
\[
|(y_1, s_1) - (x_1, t_1)|= r-ht_1.
\]
Then there exists an open set $\widetilde\Sigma\subseteq\R^{d+1}$ with $C^{1}$ boundary that is externally tangent to
$\Omega_{1}^\delta$ at $(y_1, s_1)$ in the domain  $\R^d\times [s_0,s_1]$ for some $s_0<s_1$ such that
\[
\calK^{\widetilde{\Sigma}}_{y_1,s_1}
\;\ge\;
\calK^{\Sigma}_{x_1,t_1} + h
\qquad\text{and}\qquad
\widetilde{\Sigma}(s_1) = \Sigma(t_1) + (y_1 - x_1).
\]
\item \label{L.3.3.ii}
Assume that an open set $\Sigma\subseteq\R^{d+1}$ with $C^2$ boundary is internally tangent to $\Omega_{2}^{\delta,r,h}$
at some point $(x_2, t_2) \in \partial \Omega_{2}^{\delta,r,h}$. 
Let $(y_2, s_2) \in \partial\Omega_{1}^\delta$ be such that
\[
|(y_2, s_2) - (x_2, t_2)|= r-ht_2.
\]
Then there exists an open set $\widetilde\Sigma\subseteq\R^{d+1}$ with $C^{1}$ boundary that is internally tangent to
$\Omega_{2}^\delta$ at $(y_2, s_2)$ in the domain  $\R^d\times [s_0,s_2]$ for some $s_0<s_2$ such that
\[
\calK^{\widetilde{\Sigma}}_{y_2,s_2}
\;\le\;
\calK^{\Sigma}_{x_2,t_2} - h
\qquad\text{and}\qquad
\widetilde{\Sigma}(s_2) = \Sigma(t_2) + (y_2 - x_2).
\]
\end{enumerate}
\end{lemma}

\begin{proof}
We only provide the proof of item \ref{L.3.3.i}; item \ref{L.3.1.ii} follows similarly. 

Let us first define a shift version of $\Sigma$ up to time $t_1$ by
\[
\Sigma':=\{(x,t)-(x_1,t_1)+(y_1,s_1)\,:\, (x,t)\in\Sigma,\,t\leq t_1\},
\]
and a neighborhood of $(\Sigma')^c$ as follows
\[
N:=\left\{(y+z,s+\tau)\,:\, (y,s)\in (\Sigma')^c,\, s+\tau\leq s_1,\, |(z,\tau)|\leq h(s_1-s)\right\}.
\]
We can rewrite $N$ as
\[
N=\{(y,s)\,:\, s'\leq s_1,\,|(y,s),P_{\Sigma'}(y,s)|\leq h(s_1-s')\}
\]
where $(y',s'):=P_{\Sigma'}(y,s)$ is the unique projection map as long as $s_1-s$ is small enough, due to the assumption that $\partial\Sigma'$ is $C^2$. Since $\partial \Sigma'$ is $C^2$, it follows from \cite[Theorem C]{leobacher2021existence} that the projection map $P_{\Sigma'}$ is $C^1$, and therefore, $N$ has at least $C^1$ boundary.

Now, we claim
\beq
\label{eq:NSigmac}
N\subseteq (\Omega_1^\delta)^c.
\eeq
Indeed, for any $(y,s)\in N$ and for
\[
(x,t):=(y',s')-(y_1,s_1)+(x_1,t_1)\in\Sigma^c\quad\text{and}\quad (y',s'):=P_{\Sigma'}(y,s),
\]
we have that
\[
|(y,s),(x,t)|\leq |(y,s),P_{\Sigma'}(y,s)|+|(y_1,s_1),(x_1,t_1)|\leq r-ht_1+h(s_1-s')= r-ht.
\]
Note that $t=s'-s_1+t_1\leq t_1$. Since $(x,t)\in\Sigma^c$, and since $(\Sigma(r))^c\subseteq (\Omega_{1}^{\delta,r,h}(r))^c$ for all $r\leq t_1$, we find $(x,t)\in (\Omega_{1}^{\delta,r,h})^c$. Together with
the definition of $\Omega_{1}^{\delta,r,h}$, this shows that $(y,s)\in(\Omega_1^\delta)^c$. Therefore, \eqref{eq:NSigmac} holds. 
Thus, defining $\widetilde\Sigma:=N^c$, we get that
$\Omega_1^\delta\subseteq \widetilde\Sigma$. Moreover,  the definitions imply  $\widetilde{\Sigma}(s_1) = \Sigma(t_1) + (y_1 - x_1)=\Sigma'(s_1)$.


Finally, we show $\calK^{\widetilde{\Sigma}}_{y_1,s_1}
\ge
\calK^{\Sigma}_{x_1,t_1} + h$. For the convenience of notations, let us assume that $y_1=0$ and $s_1=0$. Then,  for each $j\geq 1$, we define 
\[
\wti \Sigma_j:= \{(jx,jt)\,:\, (x,t)\in\wti\Sigma\}\quad\text{and}\quad  \Sigma_j':= \{(jx,jt)\,:\, (x,t)\in \Sigma'\}.
\]
Since $\calK^{\Sigma}_{x_1,t_1}=\calK^{\Sigma'}_{y_1,s_1}$and the blow-up procedure preserves the normal direction, we have $\calK^{\widetilde{\Sigma}}_{y_1,s_1}=\calK^{\widetilde{\Sigma}_j}_{0,0}$ and $
\calK^{\Sigma}_{x_1,t_1}=
\calK^{\Sigma'_j}_{0,0}$. Moreover, the definition of $\Sigma'$ yields 
\beq\lb{blowup}
\wti\Sigma_j=\{(x,t)\,:\, t\leq 0,\,d((x,t),(\Sigma'_j)^c)> -ht\}.
\eeq
Because both $\Sigma'$ and $\widetilde{\Sigma}$ have $C^1$ 
boundary, it follows that there exist half spaces $\Sigma'_\infty$ and $\wti\Sigma_\infty $ such that
\[
\Sigma'_j\to \Sigma'_\infty\quad\text{and}\quad\wti\Sigma_j\to \wti\Sigma_\infty 
\]
locally uniformly in the Hausdorff distance. Indeed, the boundaries of $\Sigma'_\infty$ and $\wti\Sigma_\infty $ are the tangent planes of $\Sigma'_j$ and $\wti\Sigma_j $ at $(0,0)$, respectively. Moreover, \eqref{blowup} holds with $\infty$ instead of~$j$.

\begin{figure}[t]
\label{fig1}
    \centering
    \includegraphics[scale=0.17]{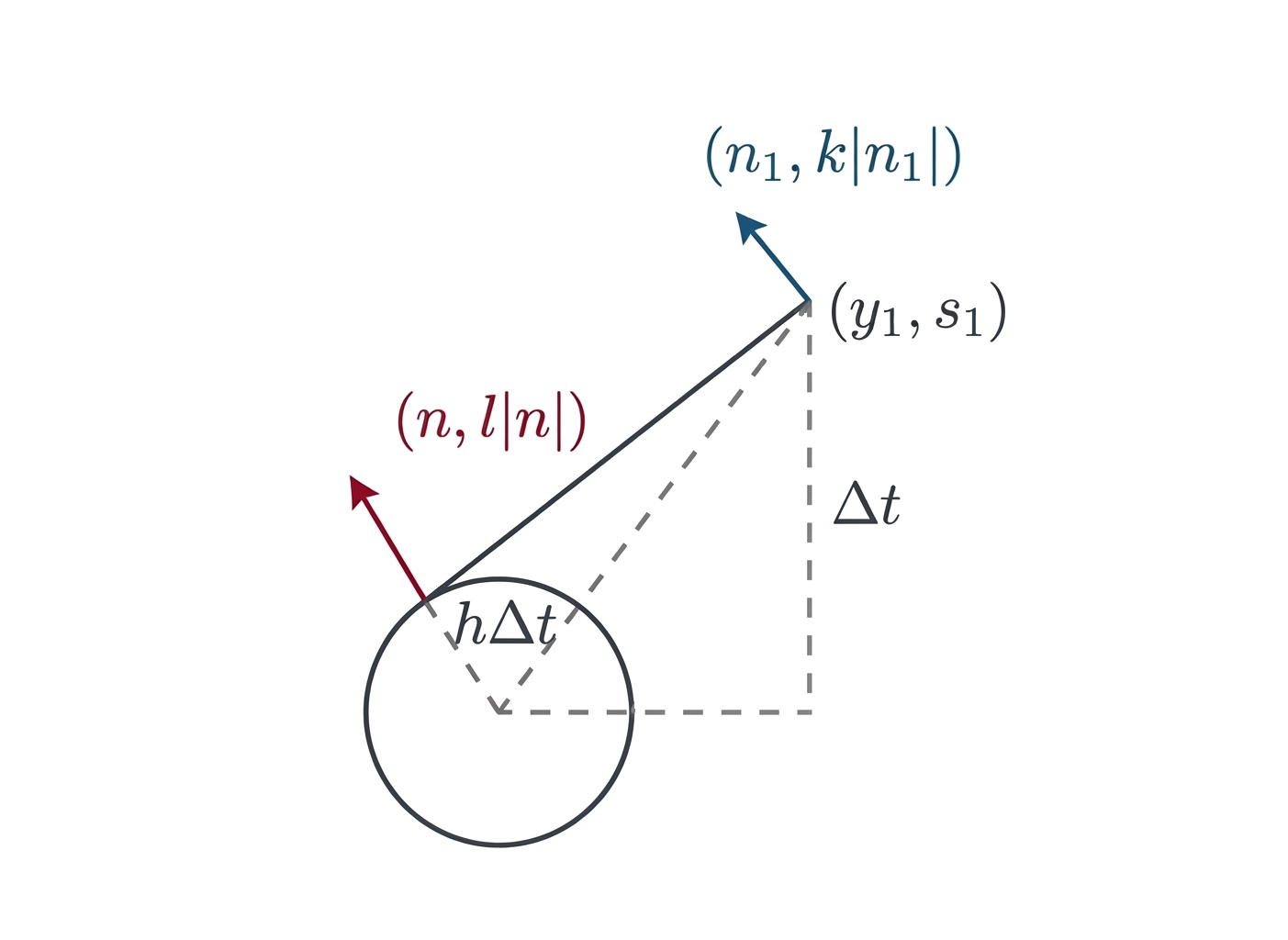}
\end{figure}

Assuming $k:=\calK^{\Sigma'}_{y_1,s_1}$ is finite, we can write
\[
\Sigma'_\infty=\{(x,t)\,:\, (x,t)\cdot(n_1,k)\ge 0\}
\]
where $(n_1,k)\in\mathbb{S}^{d-1}\times \bbR$. Similarly, assuming $l:=\calK^{\widetilde{\Sigma}}_{y_1,s_1}$ is finite, we can denote
\[
\wti\Sigma_\infty=\{(x,t)\,:\, (x,t)\cdot(n_2,l)\ge 0\}
\]
where $(n_2,l)\in\mathbb{S}^{d-1}\times \bbR$. 
Recall that $0\in \widetilde{\Sigma}(0) =\Sigma'(0)$. So, the spatial part of the normal directions at the origin are the same, i.e., $n_1=n_2=:n$.

Next, \eqref{blowup} with $\infty$ in place of $j$ yields that 
\beq\lb{3.6}
\wti\Sigma_\infty=\{(x,t)\,:\, t\leq 0,\,|(x,t),(y,s)|> -ht,\,(y,s)\in (\Sigma'_\infty)^{c}\},
\eeq
and so for any $x_1$ such that $x_1\cdot n+k\times (-1)=0$, the surface $\partial \widetilde{\Sigma}_\infty =\{(y,s) \, :\,  (y,s)\cdot (n,l)=0\}$
is tangent to the space-time ball $B_{h}((x_1,-1))$. In particular, the point $p:=(x_1, -1) +h\frac{(n,l)}{|(n,l)|}$ lies on $\partial \widetilde{\Sigma}_\infty$; equivalently, $p\cdot (n, l)=0$. Thus, we get
\[
x_1\cdot n=k,\quad x_1\cdot n-l=-h \sqrt{1+l^2}, 
\]
which yields
\[
k-l =-h \sqrt{1+l^2}\leq -h.
\]
Recall that the boundary velocities are given by $k=\calK^{{\Sigma'}}_{y_1,s_1}$ and $l=\calK^{\widetilde{\Sigma}}_{y_1,s_1}$. 
Thus, we obtain
\[
\calK^{\widetilde{\Sigma}}_{y_1,s_1}\geq \calK^{{\Sigma'}}_{y_1,s_1}+h= \calK^{{\Sigma}}_{x_1,t_1}+h.
\]
Finally, it follows from \eqref{3.6} that $l>-\infty$.

Now, let us assume that $k=+\infty$. Then
\[
\Sigma'_\infty=\{(x,t)\,:\, (x,t)\cdot(0,1)\ge 0\}.
\]
Since $(0,0)$ is on the boundaries of $\Sigma'_\infty$ and $\wti\Sigma_\infty$ and $h<1$,  \eqref{3.6} implies that 
\[
\wti\Sigma_\infty=\{(x,t)\,:\, (x,t)\cdot(0,1)\ge 0\}.
\]
This then implies that $l=+\infty$, which yields the conclusion in this case.



\end{proof}

In our proof of comparison (Theorem \ref{L.cp}), we will use a result on interposition of sets (Theorem \ref{P.3.9}) to produce  $C^{1,1}$ sets $\Sigma_1$ and $\Sigma_2$ such that $\Sigma_1$ is externally tangent to $\Omega_1^{\delta, r, h}$ and  $\Sigma_2$ is internally tangent to $\Omega_2^{\delta, r, h}$. At that point, we would like to apply Lemma \ref{L.3.3} to produce sets $\widetilde\Sigma_1$ and $\widetilde\Sigma_2$ that are internally externally tangent to $\Omega_1^\delta$ and $\Omega_2^\delta$ and then apply the fact that  $\Omega_1^\delta$ is a viscosity subflow and $\Omega_2^\delta$ is a viscosity superflow. The only problem with this plan is that Lemma \ref{L.3.3} requires $\Sigma_1$ and $\Sigma_2$ to have $C^2$, rather than merely $C^{1,1}$, boundaries. In the following lemma, we use an approximation argument to overcome this difficulty. 

Recall that, for an open set $U\subset \R^d$ with $C^{1,1} $ boundary, we use $p(\cdot;U)$ to denote the solution to \eqref{elliptic} in $U$ with $0$ boundary condition (see Definition \ref{Ass}). 

\begin{lemma}[Approximation]\lb{L.3.6}   
Let $T>0$ and $\delta, r, h\in [0,1)$, and let two bounded sets $\Omega_1$ and $\Omega_2$  be, respectively, viscosity sub- and superflows to \eqref{main} 
in $\R^d \times (0,T)$. Let $\Omega_i^\delta$ and $\Omega_i^{\delta,r,h}$, for $i=1,2$, be as defined in  \eqref{5.0}-\eqref{5.2}.
\begin{enumerate}[(i)]
    \item \lb{L.3.6.i}
Assume that an open set $\Sigma_1\subseteq\R^{d+1}$ with $C^{1,1}$ boundary is externally tangent to $\Omega_{1}^{\delta,r,h}$
at some point $(x_1, t_1) \in \partial \Omega_{1}^{\delta,r,h}$ and 
$
|\calK^{\Sigma_1}_{x_1,t_1}|<\infty$. 
Let $(y_1, s_1) \in \partial\Omega_{1}^\delta$ be such that
\[
|(y_1, s_1) - (x_1, t_1)|= r-ht_1.
\]
Then we have 
\[
(1+\delta)^{-2}V(y_1,q_1)\geq \calK^{\Sigma_1}_{x_1,t_1} +  h,
\]
where $q_1:=\nabla p(y_1;\Sigma_1(t_1)+y_1-x_1)$. 
\item \lb{L.3.6.ii} Assume that an open set $\Sigma_2\subseteq\R^{d+1}$ with $C^{1,1}$ boundary is internally tangent to $\Omega_{2}^{\delta,r,h}$
at some point $(x_2, t_2) \in \partial \Omega_{2}^{\delta,r,h}$ and $
|\calK^{\Sigma_2}_{x_2,t_2}|<\infty$. 
Let $(y_1, s_1) \in \partial\Omega_{2}^\delta$ be such that
\[
|(y_1, s_1) - (x_1, t_1)|= r-ht_1.
\]
Then we have 
\[
(1+\delta)^2V(y_2,q_2)\leq  \calK^{\Sigma_2}_{x_2,t_2} - h
\]
where $q_2:=\nabla p(y_2;\Sigma_2(t_2)+y_2-x_2)$.
\end{enumerate}
\end{lemma}

\begin{proof}
We shall only prove \ref{L.3.6.i}, as \ref{L.3.6.ii} follows similarly.

We would like to apply Lemma \ref{L.3.3}; however, since $\Sigma_1$ is only known to be $C^{1,1}$ instead of $C^2$, we need to approximate the set $\Sigma_1$. 
It follows from \cite[Theorem 3.2]{card2007} that there exist open sets $\Sigma_{1,n}$ converging to $\Sigma_1$ in the $C^{1,b}$ sense (see  \cite[Subsection 2.1]{card2007} for the definition), and $(x_{1,n},t_{1,n})\in\partial \Omega_1^{\delta, r, h}$ converging to $(x_1,t_1)$, such that $\Sigma_{1,n}$ is externally tangent to $\Omega_1^{\delta, r, h}$   at $(x_{1,n},t_{1,n})$ and, moreover, each $\Sigma_{1,n}$ is $C^2$ in a neighborhood of $(x_{1,n},t_{1,n})$. After modification outside the neighborhood (by enlarging the sets), we can just assume $\Sigma_{1,n}$ to be $C^2$. 

We let $(y_{1,n}, s_{1,n})$ be such that $|(y_{1,n}, s_{1,n}) - (x_{1,n}, t_{1,n})|= r-ht_{1,n}$.
By Lemma \ref{L.3.3}, there is an open set $\widetilde\Sigma_n\subseteq\R^{d+1}$ with $C^{1}$ boundary such that $\widetilde\Sigma_n$ is externally tangent to
$\Omega_{1}^\delta$ at $(y_{1,n}, s_{1,n})$ in the domain of $\R^d\times [s_{0,n},s_{1,n}]$ for some $s_{0,n}<s_{1,n}$, and with
\[
\calK^{\widetilde{\Sigma}_n}_{y_{1,n},s_{1,n}}
\;\ge\;
\calK^{\Sigma_{1,n}}_{x_{1,n},t_{1,n}} + h
\qquad\text{and}\qquad
\widetilde{\Sigma}_n(s_{1,n}) = \Sigma_{1,n}(t_{1,n}) + (y_{1,n} - x_{1,n}).
\]
Since $\Omega_1^\delta$  is a viscosity subflow to \eqref{main} with $V$ replaced by $(1+\delta)^{-2}V$, Definition \ref{D.1.2} and Remark \ref{rem:after visc def} yield
\beq\lb{V11}
(1+\delta)^{-2}V(y_{1,n},\widetilde q_n)\geq \calK^{\widetilde\Sigma_n}_{y_{1,n},s_{1,n}}\geq \calK^{\Sigma_{1,n}}_{x_{1,n},t_{1,n}} +  h
\eeq
where 
$\widetilde q_n:=\nabla p(y_{1,n};\widetilde\Sigma_n(s_{1,n})).
$
Since $\Sigma_{1,n}$ converges to $\Sigma_1$ in the $C^{1,b}$ sense and $\widetilde{\Sigma}_n(s_{1,n}) = \Sigma_{1,n}(t_{1,n}) + (y_{1,n} - x_{1,n})$, we know that
\[
\widetilde q_n\to \widetilde q_1=\nabla p(y_{1};\Sigma_1(t_{1})+y_1-x_1).
\]
Therefore, after passing $n\to \infty$, \eqref{V11} yields the conclusion.
\end{proof}

\subsection{Proof of comparison}
\label{ss:comp and wp}
\begin{proof}[Proof of Theorem \ref{L.cp}]

For some $h\in (0,1)$ and for  $\eps\in (0,1)$ (which will be selected independently of $h$), set
\[
r:=\eps^{2}h,\quad \delta:=\eps h.
\]
Let us use $C_V\geq 1$ to denote the constant from Assumption \ref{assump2}\ref{item:assump1.1 2}. We take $\ep$ to be small enough to ensure 
\beq
\label{delta r h}
 3r\leq \frac{\delta}{C_V} \quad \text{ and } \quad \delta\leq \frac{h}{4C_V}. 
\eeq



Let $\Omega_1^\delta$ and $\Omega_2^\delta$ be as given by \eqref{5.0} and let $\Omega_1^{\delta, r,h}$ and $\Omega_2^{\delta, r,h}$ be as in \eqref{5.1} and \eqref{5.2}. 
The goal is to  show that  $\overline{\Omega_1}\subseteq \text{Int }\Omega_2$. 
Since the choice of $\eps$ is independent of $h$ and the shape of the initial set, it suffices to show  
$\overline{\Omega_1^{\delta,r,h}}\cap (\R^d\times [0,\eps^2])\subseteq \text{Int }\Omega_2^{\delta,r,h}$, which implies 
$\overline{\Omega_1^\delta}\cap (\R^d\times [0,\eps^2])\subseteq \text{Int }\Omega_2^\delta$. The proof is finished after iteration.

By the definitions of viscosity sub- and superflow, $\Omega_1$ and $\Omega_2$ are upper- and lower- semicontinuous. Moreover, by assumption, we have $\overline{\Omega_1(0)}\subset\Omega_2(0)$. Therefore, if $h>0$ is sufficiently small (depending on the initial data), then 
\beq
\label{BrOmega1}
B_r\left({\Omega_1^{\delta, r,h}}\right)\bigcap (\R^d\times[0,t])\subseteq \Omega_2^{\delta, r,h}
\eeq
 for all sufficiently small $t>0$. We fix one such $h$.


Now, we assume for contradiction that the inclusion \eqref{BrOmega1} 
does not hold for all $t\in [0,\eps^2]$. 
Since it holds for small $t>0$, this implies that  the two sets must be ``as close as $r/2$" at some  $T'\in (0,\eps^2)$: that is, we must have that
\beq\label{r/2'}
B_{r/2}\left(\Omega_1^{\delta, r,h}\right)\bigcap (\R^d\times [0,T'])\subseteq \Omega_2^{\delta,r,h},
\eeq
 and that there exist $(x_1,t_1)\in\partial \Omega_1^{\delta, r,h}$ and $(x_2,t_2)\in\partial\Omega_2^{\delta,r,h}$ such that
\beq\lb{r/2}
|(x_1,t_1)- (x_2,t_2)|=d\left(\Omega_1^{\delta, r,h}\bigcap (\R^d\times [0,t_1]), (\Omega_2^{\delta,r,h})^c\bigcap (\R^d\times [0,t_2])\right)=r/2,
\eeq
with $T'=t_1$ or $T'=t_2$.

It follows from Lemma \ref{L.3.1} that $(x_1,t_1)$ is regular from the interior of $\Omega_1^{\delta,r,h}$ and $(x_2,t_2)$ is regular from the exterior of $\Omega_2^{\delta,r,h}$. Furthermore, \eqref{r/2} implies that $(x_1,t_1)$ is also regular from the exterior of $\Omega_1^{\delta,r,h}$ and $(x_2,t_2)$ is also regular from the interior of $\Omega_2^{\delta,r,h}$. Thus, the unique unit outer normal direction of $\Omega_1^{\delta,r,h}$ at $(x_1,t_1)$ is the same as the unique unit outer normal direction of $\Omega_2^{\delta,r,h}$ at $(x_2,t_2)$, which is given by $\nu/|\nu|$, where
\beq
\label{nu'}
\nu:=(x_2-x_1,t_2-t_1)\neq 0.
\eeq

We have now verified the hypotheses of Lemma \ref{L.3.4}; the lemma thus implies  $x_1\neq x_2$ and
\beq
\label{Ksigma<infty 2}
|\calK^{\Sigma_1}_{x_1,t_1}|,\quad |\calK^{\Sigma_2}_{x_2,t_2}|<\infty.
\eeq
We can now apply Theorem \ref{P.3.9}: indeed, \eqref{r/2'}, \eqref{r/2}, and the fact that  $x_1\neq x_2$ imply that the hypotheses of Theorem \ref{P.3.9} hold. Thus we find that
there exists a compact set $\Sigma_1\subseteq\R^{d+1} $ with $C^{1,1}$ boundary that is externally tangent to $\Omega_1^{\delta, r,h}$ at $({x}_1,{t}_1)$ and such that
\beq
\label{defSigma2}
\Sigma_2 := \Sigma_1 + ({x}_2,{t}_2) - ({x}_1,{t}_1)
\eeq
is internally tangent to $\Omega_2^{\delta,r,h}$ at $({x}_2,{t}_2)$. 
The definition of $\Sigma_2$ implies
\beq
\label{K Sigma 1 and 2}
\calK^{\Sigma_1}_{x_1,t_1}=\calK^{\Sigma_2}_{x_2,t_2}.
\eeq
This equality will be crucial in obtaining the desired contradiction and completing the proof of this theorem. 

We are now exactly in the situation of Lemma \ref{L.3.6}. Applying the lemma, we  find  
\begin{align}
\label{lem3.5 consequence}
(1+\delta)^{-2}V(y_1,q_1)\geq \calK^{\Sigma_1}_{x_1,t_1} +  h,
\qquad
(1+\delta)^2V(y_2,q_2)\leq  \calK^{\Sigma_2}_{x_2,t_2} - h,    
\end{align}
where  $(y_i, s_i)\in \Omega_i^\delta$, for $i=1,2$, are  such that 
\beq\lb{rdis}
|(y_1, s_1) - (x_1, t_1)|= r-ht_1\quad\text{and}\quad |(y_2, s_2) - (x_2, t_2)|= r-ht_2,
\eeq
and 
\[
q_1:=\nabla p(y_1;\Sigma_1(t_1)+y_1-x_1),\quad q_2:=\nabla p(y_2;\Sigma_2(t_2)+y_2-x_2).
\]

Recall that $\Sigma_2$ is a translate of $\Sigma_1$; in particular, the definition  \eqref{defSigma2} of $\Sigma_2$ implies, 
\beq\lb{3.8}
\Sigma_2(t_2)+(y_2-x_2) = \Sigma_1(t_1) +(y_1-x_1)+ (y_2 - y_1).
\eeq
So, denoting $\delta_y:=y_2-y_1$, we have that $p(x-\delta_y; \Sigma_1(t_1)+(y_1-x_1))$ solves 
\[
F(x-\delta_y,p, \nabla p,D^2p)=0\quad \text{ in }\Sigma_2(t_2)+(y_2-x_2)
\]
with $0$ boundary condition. 
By \eqref{r/2} and \eqref{rdis}, 
\begin{equation}
    \label{3r}
|\delta_y|=|y_2-y_1|\leq |y_1-x_1|+|x_1-x_2|+|x_2-y_2|\leq 3r.
\end{equation}
Recall that $\ep$ was taken sufficiently small to ensure \eqref{delta r h} holds; therefore,
\begin{equation}
\label{size delta y}
    |\delta_y|=|y_2-y_1|\leq \delta/C_V.
\end{equation}
Since $y_2$ is on the boundary of $\Sigma_2(t_2)+y_2-x_2$, it follows from  the previous line,  \eqref{3.8}, and Assumption \ref{assump3}\ref{item:assump1.1 3}  that for all $\eps>0$ sufficiently small,
\[
|\nabla_x p(y_2-\delta_y;\Sigma_1(t_1) + (y_1 - x_1))|\leq (1+\delta)|\nabla_x p(y_2;\Sigma_2(t_2)+(y_2-x_2))|.
\]
This implies that
\beq\lb{V22}
|q_1|\leq (1+\delta)|q_2|.
\eeq
Again in view of \eqref{3.8},
\beq\lb{V222}
q_1\text{ and }q_2\text{  point in the same direction}.
\eeq


Our assumption \eqref{cond5'} on $V$, together with \eqref{size delta y}, yields, 
\[
(1+\delta)^{-2}V(y_1,q_1)\leq V(y_2,(1+\delta)^{-1} q_1)+C_V|y_1-y_2|^2/\delta +C_V\delta,
\]
where $C_V\geq 1$ is the constant from Assumption \ref{assump2}\ref{item:assump1.1 2}.
Next,  \eqref{V22} and \eqref{V222}, together with our assumption  \eqref{cond5} on $V$, imply that the first term on the right-hand side of the previous line is bounded from above by $V(y_2, q_2)$.  And, upon using \eqref{size delta y} and \eqref{delta r h}, we find,
\[
C_V|y_1-y_2|^2/\delta +C_V\delta\leq 2C_V\delta\leq h/2.
\]
Thus we obtain,
\[
(1+\delta)^{-2}V(y_1,q_1)\leq V(y_2, q_2)+h/2.
\]

Next, we use \eqref{delta r h},  \eqref{cond5}, and then \eqref{cond5'}    to find,
\[
(1+\delta)^2 V(y_2,q_2)+\frac{h}{2}\geq (1+\delta)^2 V(y_2,q_2)+C_V\delta \geq V(y_2,(1+\delta)q_2)\geq V(y_2,q_2).
\]
Putting these  inequalities together we find
\[
(1+\delta)^{-2}V(y_1,q_1) \leq (1+\delta)^2V(y_2, q_2)+h. 
\]
Using \eqref{lem3.5 consequence} to bound the left-hand side of the previous line from below and the right-hand side of the previous line from above, we find,
\begin{align*}
\calK^{\Sigma_1}_{x_1,t_1} +  h\leq (1+\delta)^{-2}V(y_1,q_1)\leq (1+\delta)^2 V(y_2,q_2)+h\leq  \calK^{\Sigma_2}_{x_2,t_2}.
\end{align*}
Recalling the equality in \eqref{K Sigma 1 and 2}, we see that we have obtained the desired contradiction. Thus, we proved that $\Omega_1^{\delta, r,h}\cap (\R^d\times [0,\eps^2])\subseteq \Omega_2^{\delta,r,h}$,  finishing the proof.  
\end{proof}

\section{Existence and uniqueness}
\label{sec:existence and uniqueness}
With the comparison principle in hand, we establish several well-posedness results.
\subsection{Maximal flows}
We adapt Perron's method to viscosity flows to obtain the existence of  maximal flows and establish Theorem \ref{T.3.4}.

\begin{proof}[Proof of Theorem \ref{T.3.4}]
We define $\Omega$ to be the union of all $U$ where  $U$ is  a viscosity subflow of \eqref{main} with initial $U(0)\subseteq\Omega_0$. First, we claim that the  set of viscosity subflows is non-empty and that $\Omega$ satisfies that $\Omega(0)=\Omega_0$.
Indeed, for any spatial ball $\calB_r(x_0)\subseteq \Omega_0$ (centered at $x_0$ with radius $r$), it is direct to check that
\[
U:=\{(x,t)\,:\,t<r/A,\, x\in \calB_{r-At}(x_0)\}
\]
with $A:=\max_{x\in\overline \Omega_0} \{ -V(x,0),0\}$, is a viscosity subflow. Here we use the convention that $1/0=\infty$. Moreover, since $\Omega_0$ satisfies the interior ball condition, it follows that $\Omega(0)=\Omega_0$.

Next, by the stability of viscosity flow, $\Omega$ is again a subflow. The proof follows the same as the one for \cite[Proposition 3.7]{cardaliaguet2000front}.

Now, we aim to show that $\Omega_\#$ is a superflow. Suppose that this is not the case.  Then there exists a regular set $\Phi$ that is internally tangent to $\Omega_\#$ at some
$(x_0,t_0)\in\partial \Omega_\#\cap \partial\Phi$
such that $\nu_{x_0}\neq 0$ and
\beq \lb{5.111}
 V(x_0,\nabla p(x_0;\Phi(t_0)))> \calK^{\Phi}_{x_0,t_0}.
\eeq
Here $\nu_{x_0,t_0}=(\nu_{x_0},\nu_{t_0})$ denotes the unit outer normal direction of $\Phi$ at $(x_0,t_0)$.
Moreover, by slightly modifying $\Phi$, we can assume that 
\beq\lb{4.3}
\{(x_0,t_0)\}=\partial \Omega_\#\cap \partial\Phi\cap\{t\leq t_0\}.
\eeq
We will write $\phi(x,t)=p(x;\Phi(t))$.

Let $e_0\in \bbS^{d}$ be the outer unit normal direction of $\partial\Phi$ at $(x_0,t_0)$. For some $\eps\in (0,1)$, set
\[
\Phi_{\eps}:=\{(x,t)+\eps e_0\,:\, (x,t)\in \Phi,\,t< t_0\},
\quad\text{and}\quad
\widetilde\Omega:=\Omega_\#\cup \Phi_\eps.
\]
Since $\nu_{x_0}\neq 0$, $\Phi_\eps$ is non-empty and $\widetilde\Omega$ is strictly larger than $\Omega_\#$ for all positive $\eps$.   
Due to \eqref{4.3}, there exists $\gamma_\eps>0$ satisfying $\lim_{\eps\to 0}\gamma_\eps=0$ such that
\beq\lb{4.4}
\widetilde\Omega\subseteq \Omega\cup B_{\gamma_\eps}((x_0,t_0)).
\eeq
Denoting $\phi_\eps(x,t)=p(x;\Phi_\eps(t))$, by \eqref{5.111}, 
Assumption \ref{assump2}, and $\Phi\in C^{1,1}$, we have
\beq\lb{4.5}
\calK^{\Phi_\eps}_{x,t}<V(x,{\nabla\phi_\eps}(x,t))
\eeq
for all $(x,t)\in \partial\Phi_\eps\cap B_{\gamma_\eps}((x_0,t_0))$ when $\nu_x\neq 0$, and $\eps$ is sufficiently small.

To complete the proof, it suffices to show that $\widetilde\Omega$ is a subflow, as this will yield the desired contradiction with the definition of $\Omega$, which will in turn imply that $\Omega$ is a viscosity superflow, as desired. 
To do so, let $\widetilde \Phi$ be a regular set containing $\widetilde\Omega$. Suppose that $\partial\widetilde\Phi$ touches $\partial\widetilde\Omega$ tangentially at a free boundary  point $(x_1,t_1)\in\partial\widetilde\Omega$. If $(x_1,t_1)\in\partial \Omega$ such that $\nu_{x_1}\neq 0$, then the boundary condition holds as $\Omega$ is known to be a subflow. Otherwise, suppose $(x_1,t_1)\in\partial \widetilde\Omega\cap \partial\Phi_\eps$. By \eqref{4.4}, we have
$(x_1,t_1)\in B_{\gamma_\eps}((x_0,t_0))$. Because $\widetilde \Phi$ contains $\widetilde\Omega$ and they attach at $(x_1,t_1)$, 
\[
\calK^{\widetilde\Phi}_{x_1,t_1}\leq \calK^{\widetilde\Omega}_{x_1,t_1}= \calK^{\Phi_\eps}_{x_1,t_1}<V(x_1,\nabla \phi_\eps(x_1,t_1)),
\]
where the last strict inequality is due to \eqref{4.5}.

Again, by $\Phi_\eps\subseteq \widetilde\Phi$ and that $(x_1,t_1)\in \partial\widetilde\Phi\cap\partial \Phi_\eps$, the monotonicity of the elliptic operator $F$ yields
\[
|\nabla p(x_1;\widetilde \Phi(t_1))|\geq |\nabla \phi_\eps(x_1,t_1)|.
\] 
Therefore, the monotonicity of $V$ yields
\[
V(x_1,\nabla p(x_1;\widetilde \Phi(t_1)))\geq V(x_1,\nabla \phi_\eps(x_1,t_1))> \calK^{\widetilde\Phi}_{x_1,t_1},
\]
which implies that $\widetilde\Omega$ is a subflow.
\end{proof}

\subsection{Uniqueness if the initial set shrinks or expands}

Next, we prove uniqueness of viscosity flows if the initial set immediately expands or shrinks. Such results for the classical Hele-Shaw flow can be found in \cite[Theorem 1.8]{contact}.

\begin{theorem}[Uniqueness under shrinking/expanding assumption]
\label{thm:shrink/expand}
Under Assumption \ref{assump2}, let $\Omega_0\subseteq\R^d $ be open, bounded and having $C^{1,1}$ boundary. If
\[
\min_{x\in \partial\Omega_0}V(x,\nabla p(x;\Omega_0))>0\quad\text{or}\quad \max_{x\in \partial\Omega_0}V(x,\nabla p(x;\Omega_0))<0,
\]
then the viscosity flow is unique. 
\end{theorem}
\begin{proof}
Let $\Omega$ be a viscosity solution with initial data $\Omega_0$.
First, we show that, under the first assumption, 
\beq\lb{4.1}
\overline{\Omega_0}\subseteq\Omega(t)\quad\text{ for all $t\in (0,1]$}.
\eeq
In this case, for some $\delta,\eps>0$, define
\[
\Phi_{\delta,\eps}(0):=\{x\in\Omega_0\,:\, d(x,\partial\Omega_0)>\delta\},\qquad \Phi_{\delta,\eps}(t):=\calB_{\eps t}(\Phi_{\delta,\eps}(0)).
\]
By the regularity assumptions on $\Omega_0$ and $V$, if $\eps>0$ is sufficiently small, then for all $t\in [0,1]$ and $\delta\in (0,\eps]$,  we have $\partial \Phi_{\delta,\eps}(t)$ is $C^{1,1}$ and 
\[
\min_{x\in \partial \Phi_{\delta,\eps}(t)}V(x,\nabla p(x;\Phi_{\delta,\eps}(t)))\geq \eps.
\]
Since $\calK^{\Phi_{\delta,\eps}}_{x,t}=\eps$ on the boundary, we get that $\Phi_{\delta,\eps}$ is a viscosity subflow. Since $\overline{\Phi_{\delta,\eps}(0)}\subseteq \Omega_0$, Theorem \ref{L.cp} yields that $\Phi_{\delta,\eps}\subseteq \Omega$ for all $t\in [0,1]$. Thus, for each $t\in (0,1]$, taking $\delta=\frac12t\eps$ yields \eqref{4.1}.
Similarly, we can show that, under the second assumption,  
\beq\lb{4.2}
\overline{\Omega(t)}\subseteq\Omega_0\quad\text{ for all $t\in ( 0,1]$}.
\eeq

Now, we prove uniqueness assuming that either \eqref{4.1} or \eqref{4.2} holds. Let $\Omega_i$ with $i=1,2$ be two viscosity flows with initial data $\Omega_0$. It suffices to show that $\Omega_1\subseteq \Omega_2$. We note that, since the coefficients of \eqref{main} are time-independent, $\Omega_i'(\cdot):=\Omega_i(\cdot+s)$ is also a viscosity flow for any $s\in (0,1]$ and for $i=1,2$.

First, suppose that \eqref{4.1} holds, so that
\[
\overline{\Omega_0}\subseteq\Omega_i(s)\quad\text{ for all $s\in (0,1]$ and $i=1,2$.}
\]
 This yields $\overline{\Omega_1(0)}\subseteq\Omega_2'(0)$, so the comparison principle Theorem \ref{L.cp} implies that
\[
\overline{\Omega_1(t)}\subseteq\Omega_2(t+s)
\]
for all $t\geq 0$. By the continuity of viscosity flows and passing $s\to 0$, we obtain $\Omega_1\subseteq\Omega_2$.

Next, suppose that \eqref{4.2} holds, so that 
$
\overline{\Omega_i(s)}\subseteq\Omega_0$ for  $i=1,2$ and for  all $s\in (0,1]$. Since we have $\overline{\Omega_1'(0)}\subset \Omega_2(0)$,  Theorem \ref{L.cp} yields that
$
\overline{\Omega_1(t+s)}\subseteq\Omega_2(t)$
for all $t\geq 0$. Passing $s\to 0$ yields that $\Omega_1\subseteq\Omega_2$, which finishes the proof.


\end{proof}

 \subsection{Generic uniqueness}
Finally, following the approach of \cite{cardaliaguet2000front,card2007}, we prove a generic uniqueness result. Namely, we establish that for ``most" sets $U_0$, there exists a unique viscosity flow of \eqref{main} with initial data $U_0$. To make this precise, we first recall some definitions and facts from set theory and set-valued analysis. 

By classical set theory (see e.g., \cite{rockafellar1998variational}), if $U_n$ is a sequence of uniformly bounded, closed sets in $\R^d$, then $U_n\to U\subset\R^d$ in Hausdorff distance if and only if they converge in Kuratowski limit. Such convergence requires
\[
U=\overline{\limsup}_{n\to \infty}U_n= \overline{\liminf}_{n\to \infty}U_n
\]
where Kuratowski limits superior and inferior are defined as follows
\[
\overline{\limsup}_{n\to \infty}U_n:=\bigcap_{\eps>0}\bigcap_{k\geq 1}\bigcup_{n\geq k} \calB_\eps(U_n,\eps),\qquad \overline{\liminf}_{n\to \infty}U_n:=\bigcap_{\eps>0}\bigcup_{k\geq 1}\bigcap_{n\geq k} \calB_\eps(U_n,\eps).
\]
 
For some $r>0$, let us denote by $\calC_r$  the set of all bounded, closed subsets of $\R^d$ such that they satisfy the interior ball condition with interior balls having radius $r$. We also endow the set $\calC_r$ with the Hausdorff metric. Then $\calC_r$ is a complete separable metric space.

For any $T>0$, we also let $X_T$ denote the all bounded, closed subsets of $\R^d\times [0,T]$. After endowing $X_T$ with the Hausdorff metric,  we get that $X_T$ is a complete separable metric space. Recall that a {\it residual} of  a metric space is defined to be a countable intersection of dense open subsets \cite[Section 1.4.2]{aubin2009set}. Note that Baire's Theorem states that a residual of a complete metric space is dense \cite[Theorem 2.2]{RudinFA}. 

\begin{theorem}[Generic uniqueness]
\label{thm:generic}
Under the assumptions of Theorem \ref{T.3.4}, there exists a residual $\calC_{res}$ of $\calC_r$ such that for any $U_0\in \calC_{res}$, there exists a unique viscosity flow  of the equation \eqref{main} such that $\Omega(0)=U_0$.
\end{theorem}
\begin{proof}
Take any $T>0$ and $r>0$, and let us define a map $\calS:\calC_r\to X_T$ by the following: for any $U_0\in \calC_r$, $\calS(U_0)$ is the maximal viscosity flow with initial set $U_0$ in the time interval $[0,T]$. So,  $\calS(U_0)$ is the solution from Theorem \ref{T.3.4}. First, we will show that $\calS$ is upper semicontinuous on $\calC_r$.

To this end, let $U_0,U_n\in \calC_{r}$ be such that
\[
\overline{\limsup}_{n\to \infty}U_n=\bigcap_{\eps>0}\bigcap_{k\geq 1}\bigcup_{n\geq k} \calB_\eps(U_n,\eps)=U_0.
\]
Let $\Omega_n=\calS(U_n)$ and let $\Omega=\calS(U_0)$.
It follows from the stability of viscosity flows that 
\[
\Omega':=\overline{\limsup}_{n\to \infty}\Omega_n=\bigcap_{\eps>0}\bigcap_{k\geq 1}\bigcup_{n\geq k}  B_\eps(\Omega_n,\eps)
\]
satisfies $\Omega'(0)=U_0$ and $(\Omega')^\#$ is a viscosity subflow. By Theorem \ref{T.3.4}, we have $(\Omega')^\#\subseteq \Omega$. Since $\Omega'\subseteq (\Omega')^\#$, this shows that the mapping $\calS$ is upper semicontinuous. So, it follows from \cite[Theorem 1.4.13]{aubin2009set} that $\calS$ is actually continuous over a residual $\calC_{res}$ of $\calC_r$.

Now, take any $U_0\in\calC_{res}$. 
Using the continuity property of $\calS$, we shall show that any viscosity solution $\Omega'$ with initial data $U_0$ satisfies $\Omega\subseteq \overline{\Omega'}$, which yields the desired uniqueness. The proof is similar to the one of \cite[Proposition 3.11]{cardaliaguet2000front}. Let $U_n'\in \calC_r$ be such that $ B_{1/n}(U_n')\subseteq {U_0}$ and let $\Omega'_n$ be the maximal viscosity flow with initial data $U_n'$. It follows from Theorem \ref{L.cp} that $\Omega_n'\subseteq \Omega'$. If we further take $U_n'$ to converge to $U_0$ in Hausdorff distance, then the continuity property of $\calS$ yields that 
\[
\Omega_n'=\calS(U_n')\to \calS(U_0)= \Omega
\]
in Hausdorff distance. Since $\Omega_n'\subseteq\Omega'$ for all $n$, we get that $\Omega\subseteq \overline{\Omega'}$, which finishes the proof.
\end{proof}


\medskip

\noindent {\bf Data Availability} Data sharing is not applicable to this article as no datasets were generated or analysed during
the current study.

\medskip

\noindent {\bf Conflict of interest} The authors declare that they have no conflict of interest.



\end{document}